\documentclass[preprint,10pt]{elsarticle}

\usepackage{amssymb}
\usepackage{verbatim}
\usepackage{amsfonts}
\usepackage{graphics}
\usepackage{amsmath}
\usepackage{times}
\usepackage{appendix}
\usepackage{graphicx}
\usepackage{color}
\usepackage{enumerate}
\usepackage{fancyhdr,latexsym,amsmath,amsfonts,amssymb,amsbsy,amsthm,url}
\usepackage[hscale=0.7,centering]{geometry}
\usepackage{graphicx,epsfig}
\usepackage{amssymb,amsmath,amsthm}
\usepackage{times}
\usepackage{color}
\usepackage{mathtools}
\usepackage{breqn}
\usepackage{subfigure}
\newtheorem{theorem}{Theorem}[section]

\newcounter{somecounter}
\journal{Elsevier}

\begin{document}
  {\LARGE \bf  
\begin{center}
Attractors and long transients in a spatio-temporal \\ slow-fast Bazykin's model
\end{center}
}


\vspace*{1cm}

\centerline{\bf Pranali Roy Chowdhury$^1$, Sergei Petrovskii$^{2,3}$\footnote{Corresponding author.}, Vitaly Volpert$^{3,4,5}$, Malay Banerjee$^1$}

\vspace{0.5cm}

\centerline{ $^1$ Indian Institute of Technology Kanpur, Kanpur - 208016, India}

\centerline{ $^2$ School of Mathematics \& Actuarial Science, University of Leicester, Leicester LE1 7RH, UK}

\centerline{ $^3$ Peoples Friendship University of Russia (RUDN University), 6 Miklukho-Maklaya St,} 
\centerline{ Moscow 117198, Russian Federation}

\centerline{$^4$
Institut Camille Jordan, UMR 5208 CNRS, University Lyon 1, 69622 Villeurbanne, France}

\centerline{$^5$
INRIA
Team Dracula, INRIA Lyon La Doua, 69603 Villeurbanne, France}

\vspace{1cm}

\begin{center}
{\bf Abstract}
\end{center}

Spatio-temporal complexity of ecological dynamics has been a major focus of research for a few decades. Pattern formation, chaos, regime shifts and long transients are frequently observed in field data but specific factors and mechanisms responsible for the complex dynamics often remain obscure. An elementary building block of ecological population dynamics is a prey-predator system. In spite of its apparent simplicity, it has been demonstrated that a considerable part of ecological dynamical complexity may originate in this elementary system. A considerable progress in understanding of the prey-predator system’s potential complexity has been made over the last few years; however, there are yet many questions remaining. In this paper, we focus on the effect of intraspecific competition in the predator population. In mathematical terms, such competition can be described by an additional quadratic term in the equation for the predator population, hence resulting in the variant of prey-predator system that is often referred to as Bazykin’s model. We pay a particular attention to the case (often observed in real population communities) where the inherent prey and predator timescales are significantly different: the property known as a `slow-fast’ dynamics. Using an array of analytical methods along with numerical simulations, we provide comprehensive investigation into the spatio-temporal dynamics of this system. In doing that, we apply a novel approach to quantify the system solution by calculating its norm in two different metrics such as $C^0$ and $L^2$. We show that the slow-fast Bazykin’s system exhibits a rich spatio-temporal dynamics, including a variety of long exotic transient regimes that can last for hundreds and thousands of generations.

\vspace{1.0cm}


\vspace{1cm}

\section{Introduction}
 
Mathematical modelling of long food chain and food webs is quite complex due to the multiple interaction among trophic levels. In this regard, the resource-consumer model (or prey-predator models) have received significant attention over the last few decades as such models are building blocks of long food chain and food webs. For two species prey-predator type interaction models, the predator population can be divided into two categories: specialist and generalist predator. Whenever the predator population has an alternative food source other than the local prey, then the predator is called generalist. Whereas, if the growth of the predator population depends solely on the prey abundance, then it is called specialist predator. Many field studies \cite{Schoener,Bourlot}, laboratory experiments \cite{Klomp64} have suggested that along with the prey-predator interaction, it is also important to recognize the competition for food among the individuals of same kind. This competition is density dependent and in general modelled by a quadratic mortality rate, termed as the intraspecific competition. In particular, for the specialist predator, it reduces their growth rate apart from the natural death. In 1974, Bazykin \cite{Bazykin} extended the classical Roseinzweig-MacArthur (RM) model by including intraspecific competition or self-limitation term in predator growth equation.  Interestingly, the inclusion of this negative feedback term induces additional complexity in the dynamics of the model which can capture some realistic aspects. The global dynamics of the Bazykin model and the effect of the intraspecific competition among predators in spatial pattern formation was discussed in \cite{EPL07}. In \cite{EPL05}, the intraspecific competition is considered among predators as well as super-predators in a three trophic level food chain. They have shown with the help of numerical simulations that with varying strength of intraspecific competition the system evolves from chaotic to periodic oscillation and, thus, it has a stabilizing effect on the dynamics of the system.
  
The species belonging to different trophic level have different growth rate which can differ by few orders of magnitude. In particular, the time needed for growth of individuals mostly increases along the food chain from bottom to top \cite{Muratori89}. For example, considering the interaction between hares and lynx, phytoplankton and zooplankton, insects and birds, etc., we observe that the resource compartment has much faster growth rate compared to the species belonging to higher trophic level. With this vision, researchers were trying to explore  prey-predator models by explicitly considering the difference in timescale. This gives a new perspective of understanding the ecological interaction and long nearly-periodical population fluctuation observed in nature. The temporal variations in population densities of species was studied with the help of geometric singular perturbation theory proposed by Neil Fenichel in his seminal work \cite{Fenichel}. This theory has been used to study some prey-predator models with different timescales \cite{Muratori89,Wang19,Rinaldi92,Poggiale20}. Rinaldi and Muratori \cite{Muratori91} developed a separation principle to understand the existence of slow-fast cycles and have analyzed the periodic bursting of high and low-frequency oscillations in interacting population models with two and three-trophic levels. In this work we rescale the Bazykin model into singularly perturbed slow-fast system using a small dimensionless timescale parameter $\varepsilon$ introduced in predator growth. Apart from the stable and unstable Hopf bifurcating limit cycle, we will show that the slow-fast Bazykin model exhibits additional dynamical behavior. We further give a complete bifurcation structure to understand the bifurcation of different periodic orbits, in particular canard cycles and relaxation oscillations \cite{Krupa01b}.
  
Non-spatial models of interacting populations assume that all the individuals are distributed homogeneously over space, that is, within their habitat. But in reality, the individuals are usually distributed heterogeneously. Reaction-diffusion equations \cite{Cantrell03} provide an appropriate framework to study the effect of species heterogeneity in persistence or extinction of interacting species. It incorporates the population gradient of the species, the rate of dispersal and the interaction among other species in their natural habitat. One of the major phenomenon captured by reaction-diffusion models is pattern formation due to self-organization. The diffusion driven instability (often referred to as the Turing instability) is a homogeneity breaking mechanism which leads to the formation of various spatial patterns that are ubiquitous in nature \cite{Turing52}. It is well known that the Bazykin model satisfies Turing instability condition and hence supports the formation of stationary pattern for suitable range of parameters \cite{Avila17}.
  In the slow-fast context, the Bazykin type prey-predator model can be considered as a perturbed version of the Rosenzweig-MacArthur model. Hence, it preserves the same dynamic patterns for a reasonable range of parameters, especially beyond the Hopf-bifurcation threshold. Here we will show that the slow-fast Bazykin model exhibits Turing pattern, and the Turing threshold along with the unstable eigenmode changes with varying $\varepsilon.$ We will provide a comparative study as to how the
solution of the spatial and non-spatial system behaves in the presence or absence of slow-fast timescale. We will also show that the spatial average of the solution of the corresponding spatial model also exhibits spiking behavior forming spatio-temporal canard cycles \cite{Avitabile17}. 

Further, we will discuss the different transient dynamics observed in the spatial model before reaching the final dynamics. Apart from the long-term behavior of the ecological models, the spatial model also shows intriguing transient dynamics. The `final' asymptotic dynamics of the model might take a very long time, covering a hundred and thousands of generations of the species. Therefore, it is essential to focus on the different transient behavior of the model to study how the dynamics of the system change over long timescales \cite{Hastings04,Hastings18,Hastings94}. 

The article is organized as follows. The non-spatial model is described with and without slow-fast timescale along with a complete bifurcation structure in Section 2. Here we also give schematic representation of how the bifurcation curves and the canard and relaxation oscillation curves divide the two parametric domain. In Section 3 we provide an existence criteria and bounds for the corresponding spatial model. The steady state analysis of the system, Turing instability and the properties of the emerging Turing pattern are discussed as well. The nature of the transient and their duration are explored in Section 4. We finally give the conclusion of our work in Section 5.

 \section{Temporal model with slow-fast dynamics}
	
In this work, we consider a Bazykin type model, i.e.~a prey-predator model obtained under the following assumptions. The interaction among the species is modeled by Holling type II functional response. An intra-specific competition between the predators is taken into account by a quadratic term. Furthermore, we assume that the growth rate of prey population is significantly larger than that of the predator population. In appropriately chosen dimensionless variables, the model with a difference in timescale of the growth of the interacting species is given by the following system of equations:
	\begin{equation} \label{eq:temp_fast_sys}
	\begin{aligned}
	\frac{du}{dt} &= \nu u\Big(1-\frac{u}{\chi}\Big) -\frac{\beta uv}{1+\alpha u}:=f(u,v),\\
	\frac{dv}{dt} &= \varepsilon\Big(\frac{\beta u v}{1+\alpha u}- \eta v - \delta v^2\Big):=\varepsilon g(u,v),
	\end{aligned}
	\end{equation}
	where $\nu$ is the intrinsic growth rate of the prey, $\chi$ is the prey carrying capacity, $\beta$ represents the prey predator interaction, $\alpha$ is the amount of prey by which the predation effect is maximum, $\eta$ is the per capita death rate of predators, $\delta$ is predator
     death rate caused due to intraspecific competition and $\varepsilon$ is the timescale parameter. We assume that parameter $\varepsilon$ is small, $0<\varepsilon\ll 1$, that is the dynamics of prey is much faster than that of predator. The variables $u$ and $v$ denotes the dimensionless prey and predator density at time $t.$
     
     Depending on the parameter values, the model can admit at most five equilibrium points (cf.\cite{EPL07} for details). Throughout our paper, we choose the parametric domain such that system (\ref{eq:temp_fast_sys}) has unique equilibrium point in the first quadrant. The extinction and prey only equilibrium points of the system (\ref{eq:temp_fast_sys}) are given by $E_0=(0,0),$ and $E_{\chi}=(\chi,0)$ respectively.  The coexistence equilibrium points $(u_*,v_*)$ are such that $u_*$ is the root of the cubic equation   
	\begin{equation}
	  \alpha^2\delta\nu u^3+ \delta \nu \alpha(2-\alpha \chi) u^2+ (\beta^2 \chi -\eta \alpha \beta \chi -2\delta\nu\chi\alpha+\delta\nu)u-\chi(\delta\nu+\beta\eta) = 0,  
	\end{equation} and
	\begin{equation*}
	    v_* = \frac{1}{\delta}\Big(\frac{\beta u_* }{1+\alpha u_*}- \eta\Big).
	\end{equation*}
We set 
  \begin{equation*}
    A:=\delta^2 \nu^2\alpha ^2 (2-\alpha \chi)^2-3\alpha^2\delta\nu (\beta^2 \chi -\eta \alpha \beta \chi -2\delta\nu\chi\alpha+\delta\nu),
    \end{equation*}
   \begin{equation*}
        B:=2\delta^3\nu^3 \alpha^3(2-\alpha \chi)^3-9\alpha^3\delta^2\nu^2(2-\alpha \chi)(\beta^2 \chi -\eta \alpha \beta \chi -2\delta\nu\chi\alpha+\delta\nu)+27\alpha^4\delta^2\nu^2\chi(\delta\nu+\beta\eta)
  \end{equation*}
  and consider $\Delta:=B^2-4A^3.$ From the analysis of the number of roots in a cubic equation \cite{Wang19} we obtain that if $\Delta>0$ and $\beta u_*>\eta(1+\alpha u_*)$ then the system has a unique feasible coexistence equilibrium $E_*$. Evaluating the Jacobian matrix at $E_0$ we infer that $E_0$ is a saddle and at $E_{\chi}$, the Jacobian matrix takes the form 
	\begin{equation*}
	J_{E_{\chi}}=
	\begin{pmatrix}
	-\nu&\frac{-\beta \chi}{1+\alpha \chi}\\0&\varepsilon\Big(\frac{\beta \chi}{1+\alpha \chi}-\eta\Big)
	\end{pmatrix}.
	\end{equation*}
	Thus, $E_{\chi}$ is stable if $\frac{\beta \chi}{1+\alpha \chi}<\eta$ and saddle if $\frac{\beta \chi}{1+\alpha \chi}>\eta.$ The coexistence equilibrium point cannot be obtained explicitly, so we numerically study the stability of $E_*$. We linearize the system around the equilibrium point $E_*(u_*,v_*)$ and obtain the Jacobian matrix as follows
	\begin{equation}\label{jacobstar}
	J_{E_{*}}=
	\begin{pmatrix}
	\nu\Big(1-\frac{2u_*}{\chi}\Big)-\frac{\beta v_*}{(1+\alpha u_*)^2}&\frac{-\beta u_*}{1+\alpha u_*}\\\frac{\varepsilon \beta v_*}{(1+\alpha u_*)^2}&\varepsilon\Big(\frac{\beta u_*}{1+\alpha u_*}-\eta-2\delta v_*\Big)
	\end{pmatrix}\equiv \begin{pmatrix}
	a_{11} & a_{12} \\
	a_{21} & a_{22} \\
	\end{pmatrix}.
	\end{equation}
 Considering $\chi$ as the bifurcation parameter, we have Trace($J_{E_*})=0$ and Det($J_{E_*})>0$ at the Hopf threshold $\chi=\chi_H.$ The Hopf bifurcation curve is shown in Fig.~\ref{fig:bifurcation_Hopf} in $\delta-\chi$ plane for $\nu=10,\ \alpha=1,\ \beta=2.85,\ \eta=1,\ $ and $\varepsilon=1$ fixed. Along this bifurcation curve, the system encounters a generalized Hopf bifurcation point (GH point) where the first Lyapunov number $l_1$ for Hopf bifurcation vanishes. The lower branch of the Hopf curve is super-critical where $l_1<0$ and the upper branch is sub-critical where $l_1>0$ and this transition takes place at GH point where $l_1=0.$ Thus, for a fixed $\delta$, as we move along the $\chi-$axis we observe different dynamics of the system (\ref{eq:temp_fast_sys}). For $\delta=0.11$ and $\chi=3.5,$ the coexistence equilibrium point $E_*$ is globally stable, that is all the trajectories approach this point. $E_*$ loses its stability as $\chi$ crosses the lower branch of Hopf curve. Inside the parabolic region, $E_*$ is unstable. It is surrounded by a stable limit cycle, and with increasing value of $\chi$ the size of the stable limit cycle increases. This is shown in the Fig.~\ref{fig:bifurcation_Hopf} for $\delta=0.11,$ and $\chi=4.5$ as marked by blue dot. On crossing the upper branch of Hopf curve, $E_*$ becomes stable, surrounded by an unstable limit cycle which again is surrounded by a stable limit cycle. The unstable limit cycle acts as a separatrix between the basin of attraction of the interior stable equilibrium and the outermost stable limit cycle. In a very small parametric domain near the Hopf threshold, the size of the unstable limit cycle increases whereas the size of the stable limit cycle decreases. This is verified by taking $\delta=0.11,$ and $\chi=12.25$ as shown in Fig.~\ref{fig:bifurcation_Hopf}. These two cycles coalesce and disappear at the saddle-node bifurcation point of limit cycles, for $\chi_{SN}=12.2522993$. The broken line in the Fig.~\ref{fig:bifurcation_Hopf} represents the curve of saddle-node bifurcation of limit cycle in $\delta-\chi$ plane.  
	
	\begin{figure}[ht!]
	    \centering
	    \includegraphics[width=12cm,height=7cm]{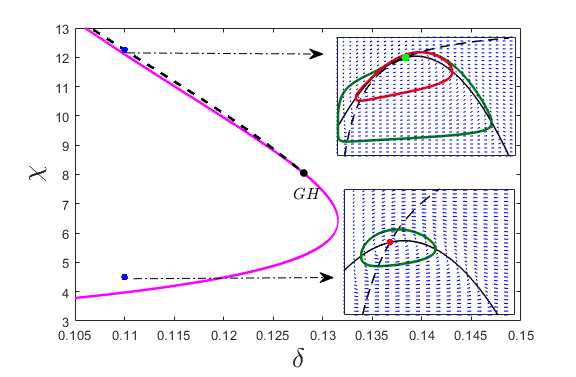}
	    \caption{Two-parametric bifurcation diagram of system (\ref{eq:temp_fast_sys}) for $\nu=10,\ \alpha=1,\ \beta=2.85, \eta=1$ and $\varepsilon=1$. The solid magenta curve represents the Hopf curve, GH (black dot) is the generalized Hopf bifurcation point on the Hopf curve. The broken line represents the saddle-node bifurcation curve of limit cycles. The local dynamics of the system is shown in the inset for two different points in $\delta-\chi$ parameter plane as mentioned in the text. The non-trivial prey and predator nullcline are shown in the black solid curve and the black broken curve. The stable and unstable attractors are marked by green and red colour. That is, the stable limit cycle (equilibrium) is represented by the green curve (dot), whereas the unstable limit cycle (equilibrium) is represented by the red curve (dot).}
	    \label{fig:bifurcation_Hopf}
	\end{figure}
	
The aim of this paper is to further investigate how the system behaves for $0<\varepsilon\ll 1$. To study the dynamics of the slow-fast system (\ref{eq:temp_fast_sys}), for sufficiently small values of $\varepsilon$ we re-write the system in terms of slow time $\tau$ as $\tau:=\varepsilon t$, where $0<\varepsilon\ll 1$ and $t$ is the fast time
\begin{equation} \label{eq:temp_slow_sys}
	\begin{aligned}
	\varepsilon\frac{du}{d\tau} &= \nu u\Big(1-\frac{u}{\chi}\Big) -\frac{\beta uv}{1+\alpha u}:= f(u,v),\\
	\frac{dv}{d\tau} &= \Big(\frac{\beta u v}{1+\alpha u}- \eta v - \delta v^2\Big):= g(u,v).
	\end{aligned}
\end{equation}
For $\varepsilon\rightarrow 0$ in system (\ref{eq:temp_fast_sys}) we obtain the corresponding fast subsystem (layer system) as follows
\begin{equation} \label{eq:fast_subsystem}
	\begin{aligned}
	\frac{du}{dt} &= f(u,v)=\nu u\Big(1-\frac{u}{\chi}\Big) -\frac{\beta uv}{1+\alpha u},\\
	\frac{dv}{dt} &= 0,
	\end{aligned}
\end{equation}
and taking $\varepsilon\rightarrow0$ in system (\ref{eq:temp_slow_sys}) we obtain a differential-algebraic system of equations (or slow subsystem) as follows
\begin{equation}\label{eq:DAE}
	\begin{aligned}
	0 &= f(u,v)= \nu u\Big(1-\frac{u}{\chi}\Big) -\frac{\beta uv}{1+\alpha u},\\
	\frac{dv}{d\tau} &= g(u,v)= \Big(\frac{\beta u v}{1+\alpha u}- \eta v - \delta v^2\Big).
	\end{aligned}
\end{equation}
The solution of system (\ref{eq:DAE}) is constrained to the set $\{(u,v) \in \mathbb{R}^2_+:f(u,v)=0\}$ known as critical manifold $C_0$, which can be written as $C_0 = C_0^0 \cup C_0^1$ where $C_0^0 = \{(u,v): u=0, v\ge 0\}$ and $C_0^1= \{(u,v): v=\frac{\nu}{\beta} \Big(1-\frac{u}{\chi}\Big)(1+\alpha u):=F(u) \}.$ The fold point $(u_f,v_f)$ of the critical manifold $C_0^1$ satisfies the following conditions 
$$ F(u_f)=0,\ F'(u_f)=0,\ F''(u_f)\ne0.$$ Furthermore, we assume that $g_u(u,v)\ne0$ such that the nullcline $g(u,v)$ transversally intersects the critical manifold $C_0^1.$ From the above conditions the coordinates of the fold point can be obtained explicitly as $u_f=\frac{1}{2\alpha}(\alpha\chi-1)$, and $v_f=F(u_f)$. This point divides the critical manifold $C_0^1$ into attracting sub-manifold $C^{1,a}_0=\{(u,v): (u,v)\in C^1_0, F'(u)<0 \}$, and repelling sub-manifold $C^{1,r}_0=\{(u,v): (u,v)\in C^1_0, F'(u)>0 \}$.
The slow flow on either branch of the critical manifold $C^1_0$, i.e $v=F(u)$ is given by (\ref{eq:DAE}). Since $v=F(u),$ then $	\frac{dv}{d\tau} = F'(u)\frac{du}{d\tau}$ and from (\ref{eq:DAE}) we have
	\begin{equation}\label{eq:slow-flow}
	\begin{aligned}
	\frac{du}{d\tau}=\frac{g(u,F(u))}{F'(u)}.
	\end{aligned}
	\end{equation}
	At the fold point we have $F'(u)=0$, therefore the flow in the neighborhood of the fold point depends on whether $g(u,F(u))=0$ or $g(u,F(u))\ne 0$ as follows:\\
	(a) If $g(u,F(u))\ne0$ at $u=u_f$, then (\ref{eq:slow-flow})
	is singular at $(u_f,v_f)$ and in this case the fold point is called jump point. The trajectory of the system (\ref{eq:temp_fast_sys}) for $\varepsilon>0$ follows the attracting slow sub-manifold closely, passes through the vicinity of the fold point and then jumps to the trivial slow manifold $(C_0^0)$ through fast horizontal flow. \\
	(b) If $g(u,F(u))=0$ at the fold point $(u_f,v_f),$ then $\delta_f = \frac{1}{v_f}\Big(\frac{\beta u_f}{1+\alpha u_f}-\eta\Big).$ Thus, 
	\begin{equation*}
	\begin{aligned}
	\frac{du}{d\tau} &= \frac{g(u,F(u))}{F'(u)}\\
	&= \frac{g(u_f,F(u_f))+(u-u_f)g_u(u_f,F(u_f)+O(u-u_f)^2}{(u-u_f)F''(u)+O(u-u_f)^2}\\
	& = \frac{g_u(u_f,F(u_f))+O(u-u_f)}{F''(u)+O(u-u_f)}.
	\end{aligned}
	\end{equation*}
	Since $g_u\ne 0$ and $F''(u)\ne 0$ at the fold point, therefore, for $\varepsilon\ll1$, the slow flow is defined at this point. The trajectory crosses the fold point and stays near the repelling sub-manifold for a certain time before jumping to the trivial manifold. In this case the fold point is called canard point and the solution of system (\ref{eq:temp_fast_sys}) passing through this point is called canard solution.
	
The critical manifold $C_0$ and the fold point on $C_0$ is independent of the time scale parameter $\varepsilon$. The coexistence equilibrium point $E_*$ of the system (\ref{eq:temp_fast_sys}) is independent of $\varepsilon$, but the stability of $E_*$ depends on $\varepsilon$. Since the explicit expression of the equilibrium point and the Hopf bifurcation threshold cannot be calculated analytically, so we numerically find them. In the $\delta-\chi$ parametric plane the Hopf curve lies at a $O(\varepsilon)$ distance from the curve of canard point, which will be called as fold curve throughout our paper to avoid ambiguity. Thus, for a fixed value of $\chi$, $\delta_H\rightarrow \delta_f$ as $\varepsilon\rightarrow 0$. This is represented in Fig.~\ref{fig:bifurcation_schematic}(a) where the Hopf curves for three values of $\varepsilon$ are shown along with the fold curve.
	
\begin{figure}[ht!]
	    \centering
	    \mbox{\subfigure[]{\includegraphics[width=8cm,height=6cm]{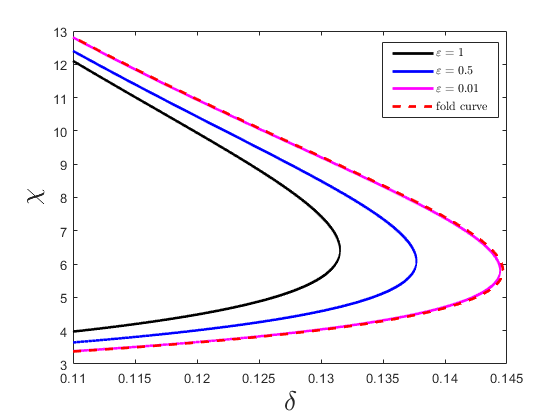}}
	    \subfigure[]{\includegraphics[width=8cm,height=6cm]{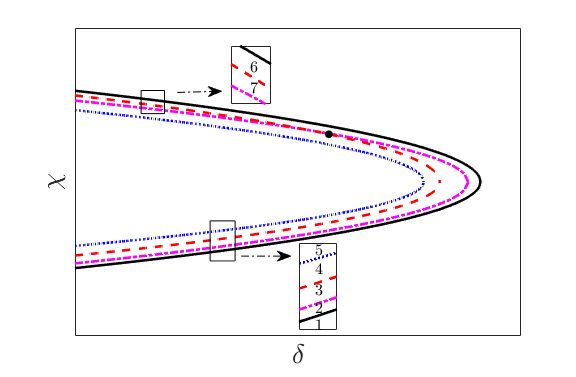}}
	    }
	    \caption{(a) The Hopf bifurcation curves for $\nu=10,\ \alpha=1,\ \beta=2.85,\ \eta=1$ and $\varepsilon=1$ (black), $\varepsilon=0.5$ (blue), $\varepsilon=0.01$ (magenta) and the curve of canard point in $\delta-\chi$ plane; (b) The schematic bifurcation diagram for $\varepsilon=0.01$ and keeping all the parameters fixed as in (a) where Hopf bifurcation curve (magenta), curve of canard point (black), the maximal canard curve (red) and the relaxation oscillation curve (blue) is shown. The curves divide the parameter plane into seven domains with different properties, see details in the text.}
	    \label{fig:bifurcation_schematic}
	\end{figure}
The intriguing slow-fast dynamics of system (\ref{eq:temp_fast_sys}) for $\varepsilon \ll 1$ can be seen in extremely narrow parametric regimes. For that we fix $\varepsilon=0.01$ and numerically obtain the maximal canard curve and the relaxation oscillation curve in $\delta-\chi$ plane. These curves along with the fold curve and Hopf bifurcation curve divides the parametric plane into seven narrow domains where different dynamics of the system are observed. To represent this we provide a schematic diagram in Fig.~\ref{fig:bifurcation_schematic}(b). The fold curve is denoted by a solid line (black) and near to this is the Hopf bifurcation curve (broken magenta curve). The maximal canard curve shown (broken red line) transversely cuts the Hopf curve at the GH point (black dot) and passes closely to the fold curve for higher values of $\chi.$ Finally, the blue dotted line is the curve of relaxation oscillation cycle which lies entirely in the parabolic region. 
	
The entire domain outside the parabolic fold curve is denoted by Domain 1, and the entire domain enclosed by the relaxation oscillation curve is denoted by Domain 5. In Domain 1, the system has unique coexistence equilibrium point $E_*(u_*,v_*)$ such that $u_*>u_f$. The equilibrium point lies on the normally hyperbolic attracting sub-manifold of the critical manifold $C_0^1$, and $E_*(u_*,v_*)$ is globally asymptotically stable. In Domain 2, we have $u_H \approx u_f$. From the stability analysis we obtain that the eigenvalues are complex conjugate with negative real part of the order of $10^{-4}$. Thus, the trajectory converges to the coexistence equilibrium point with extremely slow rate of convergence.  Also, for $\varepsilon$ sufficiently small the Hopf bifurcation curve coincides with the fold curve. Keeping all the parameters fixed as in Fig.~\ref{fig:bifurcation_schematic}(a), we take values from Domain 2 such that $\chi=4.28,\ \delta=0.1347.$  Then $u_H = 1.637,\ u_f = 1.64$ and the eigenvalues evaluated at $E_*$ are $-9.5\times 10^{-5}\pm 0.2i.$ Domain 3 is enclosed by the super-critical Hopf bifurcation branch, the maximal canard curve and the GH point, whereas, Domain 7 is enclosed by the sub-critical Hopf bifurcation branch, the maximal canard curve and the GH point. The distinctive difference between these two domains is the number of slow-fast cycles present and the stability of these cycles. The Hopf bifurcation is necessary for the existence of canard cycles (with and without head). We will discuss about the stability of these cycles but before that let us describe all the possible slow-fast cycles of system (\ref{eq:temp_fast_sys}).  
	
Let $s^*=v_f-v_{ext}$, where $v_{ext}$ is the v-coordinate of the point of exit of the slow trajectory from the critical manifold $C^0_0$ and $v_f$ is the v-coordinate of the fold point. Using similar approach as discussed in \cite{Chowdhury21} we can find $v_{ext}$. We define the continuous family of singular slow-fast cycles $\Gamma(s)$ for $s\in [0,s^*]$. Let $u_l(s)<u_r(s)$ are the two distinct roots of $F(u) = v_{ext}+s$. We define $u_r$ such that $F(u_{r}) = v_{ext}$ and $u_l(0)=u_l$, then  \\
	
\noindent (i) the cycles $\Gamma(s) = \{(u,F(u)): u\in[u_l(s),u_r(s)]\}\cup \{(u,v_{ext}+s):u\in[u_l(s),u_r(s)]\},$ for  $s\in (0,s^*) $ corresponds to canard without head,\\
	
\noindent (ii) the cycles $\Gamma(s) = \{(u,F(u)): u\in[u_l(s),u_r]\}\cup \{(u,v_{ext}+s):u\in[u_l,u_l(s)]\}\cup \{(u_l,v):v\in[v_{ext}+s,v_{ext}] \} \cup \{(u,v_{ext}):u\in[u_l,u_r]\},$ for  $s\in (0,s^*) $ corresponds to canard with head, and\\
	
\noindent (iii) the cycles $\Gamma(s) = \{(u,F(u)): u\in[u_f,u_r]\}\cup \{(u,v_f):u\in[u_l,u_f]\}\cup \{(u_l,v):v\in[v_{ext},v_f]\}\cup \{(u,v_{ext}):u\in[u_l,u_r]\},$ for  $s\in (0,s^*) $ corresponds to relaxation oscillation cycle. 

The cycles of Type (i) and (ii) are illustrated in Fig.~\ref{fig:canard_schematic}.
	\begin{figure}[ht!]
	    \centering
	    \mbox{\includegraphics[scale=0.6]{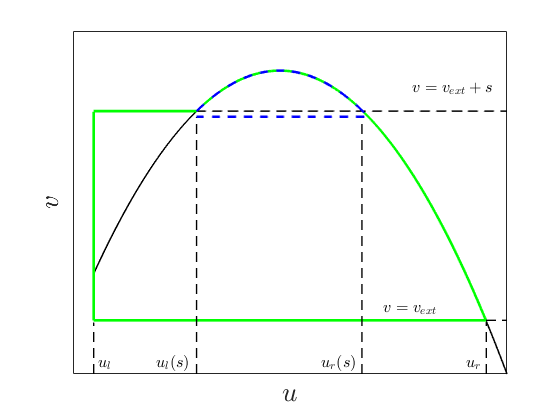}}
	    \caption{Schematic representation of the singular slow-fast cycles in the system  where Type (i) cycle is marked with blue (broken line), and Type (ii) cycle is marked with green. The canard cycle (with and without head) obtained for the system (\ref{eq:temp_fast_sys}) are perturbation of the singular cycles.}
	    \label{fig:canard_schematic}
	\end{figure}
	
	To explain the notion of stability of the canard cycles we follow \cite{Maesschalck15,Wang19} to define slow divergence integral. The integrand is given by the divergence of the vector field along the critical (slow) manifold. For the cycles of type (i) we define the slow divergence integral as 
	\begin{equation*}
	    I(s) =  \int_{u_r(s)}^{u_l(s)} \frac{\partial f}{\partial u}(u,F(u))\frac{F'(u)}{g(u,F(u))} \,du 
	\end{equation*}
and for the cycles of type (ii) we set
   \begin{equation*}
	    \Tilde{I}(s) =  \int_{u_r}^{u_l(s)} \frac{\partial f}{\partial u}(u,F(u))\frac{F'(u)}{g(u,F(u))} \,du + \int_{v_{ext}+s}^{v_{ext}} \frac{\frac{\partial f}{\partial u}(0,v)}{g(0,v)}  \,dv. 
	\end{equation*}
In order to proceed further, we recall the following theorem (see \cite{Maesschalck15,Wang19} and references therein).	
\begin{theorem}\cite{Maesschalck15,Wang19}
    For sufficiently small $\varepsilon$, the stability of the perturbed canard cycles depends on the sign of the slow divergence integral. If $I(s)<0 \ (or >0),$ then the canard cycle without head is stable (or unstable) and if $\Tilde{I}(s)<0 \ (or >0),$ then the canard cycle with head is stable (or unstable).
\end{theorem}

Let $\Gamma(s)$ be the slow-fast cycle of type (i), that is canard cycle without head and let $T$ be the point of intersection of $C_0^0$ and $C_0^1$ having the coordinate $\big(0,\frac{\chi}{\beta}\big)$. Then for any $v  \in \big(\frac{\chi}{\beta},v_f\big)$, where $v=v_{ext}+s$ the slow divergence integral $I(s)$ is defined along the critical manifold $v=F(u)$. The two distinct feasible roots of $F(u)=v_{ext}+s$ are given by $$u_r(s),\ u_l(s)= \frac{1}{2}\big(\chi-\frac{1}{\alpha}\big)\pm \frac{1}{2\alpha\nu}\sqrt{\nu^2(\alpha\chi+1)^2-4\alpha\beta\nu\chi(v_{ext}+s)}.$$
Now, depending on whether $I(s)<0$ or $>0$, the canard cycle without head is either stable or unstable. Whereas, if $\Gamma(s)$ is the canard cycle of type (ii) then the integral is calculated by dividing into two parts, one along the critical manifold $C^1_0$ and another one along $C_0^0.$

In Domain 3, the coexistence equilibrium point loses its stability via super-critical Hopf bifurcation, and the family of canard cycles without head $\Gamma(s)$, as defined above, are attracting for $\varepsilon\ll 1.$ Starting from the Hopf bifurcation curve, small amplitude canard cycle develops into large amplitude canard cycle with head (in Domain 4) and further to relaxation oscillations (in Domain 5) in an exponentially narrow interval giving rise to canard explosion. In Fig.~\ref{fig:canard_super-critical} stable canard cycles without head (cyan) emerges for the parameter values taken from Domain 3 which changes to canard with head (green) as we move from Domain 3 to Domain 4 and further to relaxation oscillations as we shift to Domain 5 (red). Domain 4 is characterized by the stable canard cycles $\Gamma(s)$ of type (ii), that is canard with head.

	\begin{figure}[ht!]
		\centering
		\includegraphics[width=10cm,height=7cm]{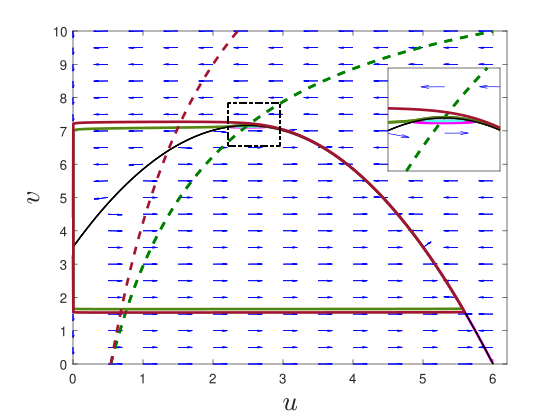}
		
		\caption{Family of canard cycles for $\nu=10,\ \chi=6,\ \alpha=1,\ \beta=2.85, \eta=1,\ \varepsilon=0.01$ for three different values of $\delta.$ Canard cycle without head for $\delta=0.14443$ (cyan), canard cycle with head for $\delta=0.14442$ (green), and relaxation oscillation cycle for $\delta=0.1444$ (red). The black solid curve represents the prey nullcline. The predator nullclines are shown by dashed curve for $\delta=0.14443$ (cyan), $\delta=0.14442$ (green), and $\delta=0.1444$ (red). The cyan and the green curve overlaps and thus cannot be distinguished in the Figure.}
		\label{fig:canard_super-critical}
	\end{figure}
	
	In Domain 7, the coexistence equilibrium point regains its stability and small unstable canard cycle without head is born which is surrounded by stable canard cycle with head. Theorem 8.4.3 of \cite{Kuehn15} states that whenever the Hopf bifurcation is sub-critical, there exists a unique parameter value where the family $\Gamma(s)$ undergoes a saddle node bifurcation of limit cycles. For the parameter values taken from Domain 7, i.e for $\nu=10,\ \chi=13,\ \alpha=1,\ \beta=2.85, \eta=1,\ \varepsilon=0.01,$ the sub-critical Hopf bifurcation threshold $\delta_H=0.109049,$ and the stable equilibrium point $E_*$ is surrounded by an unstable canard cycle without head which on the other hand is surrounded by a stable canard cycle with head (cf. Fig.~\ref{fig:canard_sub-critical}). The cycles coalesce at a saddle node bifurcation of limit cycle at $\delta_{SNC}=0.01090658,$ after which in Domain 6 the system settles down to a globally stable equilibrium point. 
    
	\begin{figure}[ht!]
		\centering
		\mbox{\subfigure[$\delta=0.1090657$]{\includegraphics[scale=0.5]{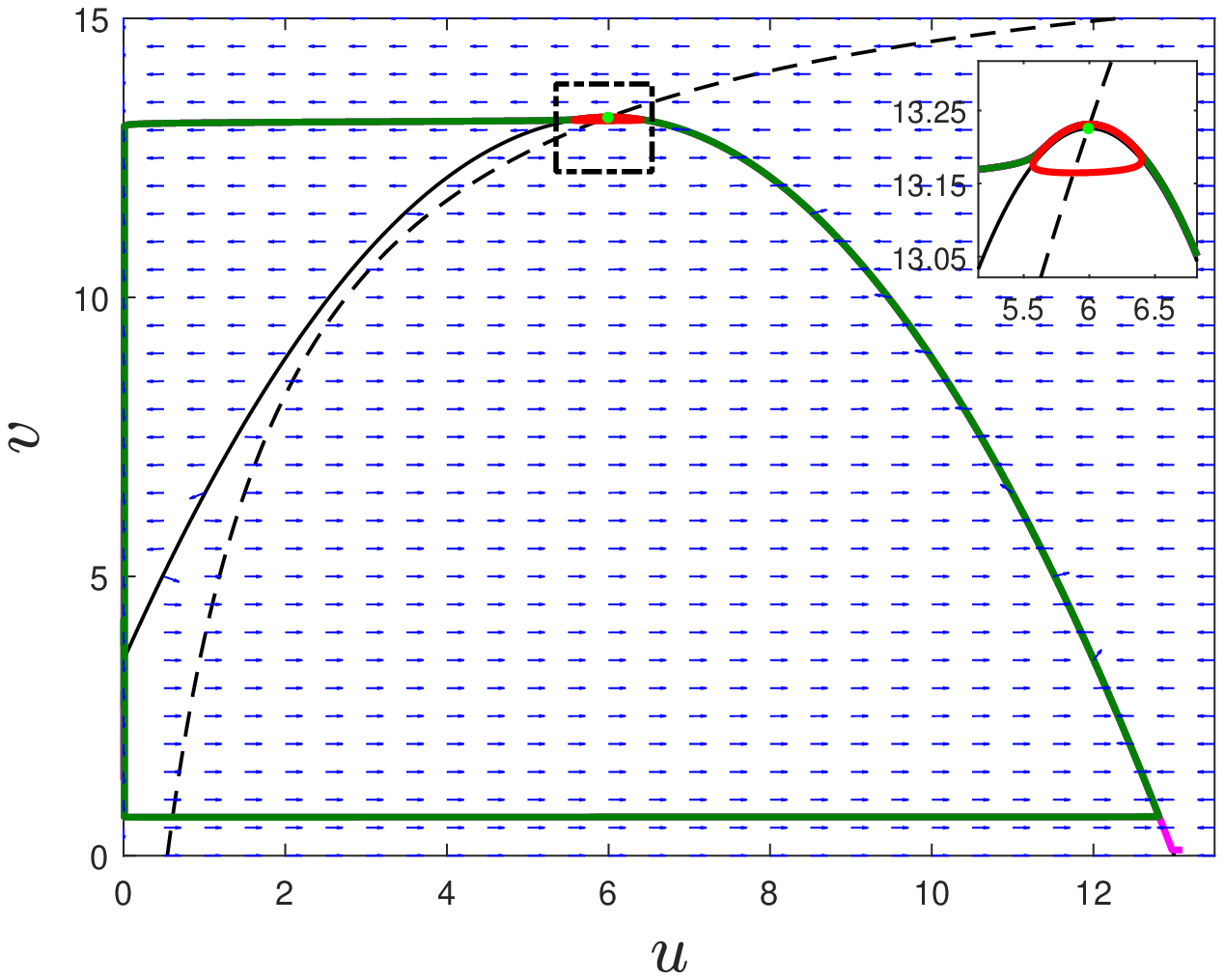}}
		\subfigure[$\delta=0.10904$]{\includegraphics[scale=0.5]{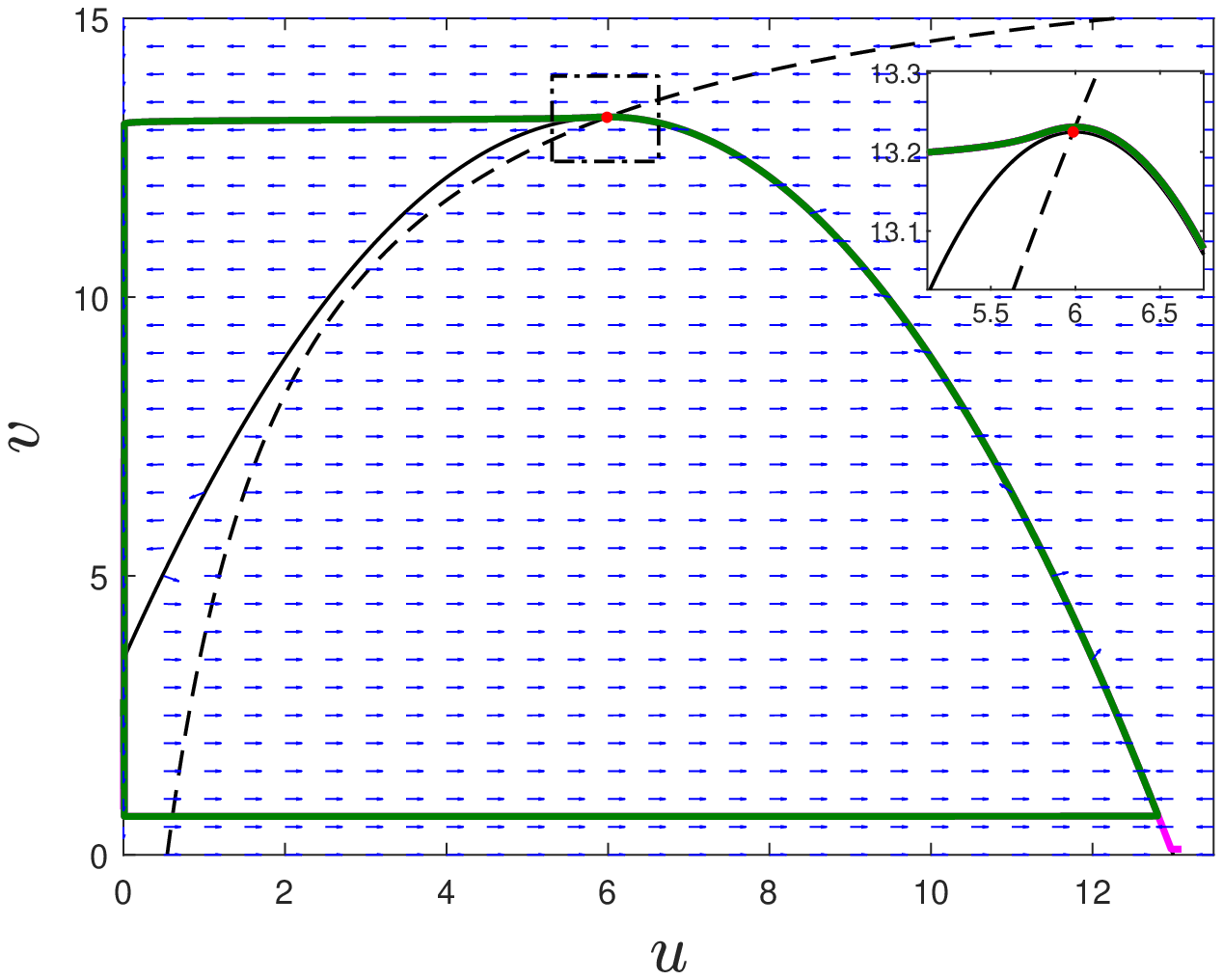}}} \caption{(a) Unstable canard cycle without head (red) and stable canard cycle with head (green) at $\delta=0.1090657,$ inset showing zoomed figure near the equilibrium, (b) Stable canard cycle with head before relaxation oscillation (green) at $\delta=0.10904,$ other parameters are fixed at $\nu=10,\ \chi=13,\ \alpha=1,\ \beta=2.85, \eta=1, \text{and \ } \varepsilon=0.01.$ Green and red dots correspond to stable and unstable equilibrium point, respectively.} \label{fig:canard_sub-critical}
	\end{figure}

The system (\ref{eq:temp_fast_sys}) therefore exhibits two different scenarios of canard explosion. One where small amplitude stable canard cycles grows to large amplitude stable relaxation oscillations in an extremely small range of $\delta$ (see Fig.~\ref{fig:canard_super-critical}). And second, where the fast transition occurs from small amplitude unstable canard cycle to large amplitude stable canard cycle and further to stable relaxation cycle (see Fig.~\ref{fig:canard_sub-critical}).
 
\section{Spatio-temporal model with slow-fast dynamics}

We now consider the corresponding spatio-temporal slow-fast model:
	\begin{equation}\label{STmodel}
	\begin{aligned}
	  \frac{\partial u}{\partial t} &= \frac{\partial ^2u}{\partial x^2} + \nu u\Big(1-\frac{u}{\chi}\Big) -\frac{\beta uv}{1+\alpha u},\\
	  \frac{\partial v}{\partial t} &= d\frac{\partial ^2v}{\partial x^2} + \varepsilon\Big(\frac{\beta u v}{1+\alpha u}- \eta v - \delta v^2\Big),
	\end{aligned}
	\end{equation}
(in dimensionless variables) over a bounded one dimensional spatial domain $\Omega=[0,L]$. Equations (\ref{STmodel}) are complemented with the initial conditions: 
	\begin{equation}\label{BC}
	    u(0,x)\,=\,u_0(x)\,\geq\,0,\,\,v(0,x)\,=\,v_0(x)\,\geq\,0,\,\ x\in \Omega,
	\end{equation}
and zero-flux boundary condition,
	\begin{equation}\label{IC}
	    u_x(t,0)\,=\,u_x(t,L)\,=\,v_x(t,0)\,=\,v_x(t,L)\,=\,0,\,\,t\,>\,0, 
	\end{equation}
corresponding to a closed ecosystem.

\begin{theorem}
    Suppose that all the dimensionless parameters involved in the reaction part are non-negative and $d>0.$ If $u_0(x)\ge 0, v_0(x) \ge 0,$ then the reaction-diffusion system (\ref{STmodel}-\ref{IC}) has a unique non-negative solution $(u(x,t),v(x,t))$ for $t>0$ and $x \in \Omega.$
\end{theorem}
\begin{proof}
We have already defined the nonlinear terms of (\ref{STmodel}) as $f(u,v)$ and $g(u,v).$ Both $f$ and $g$ are $C^1$ functions. Clearly, $(0,0)$ is a lower solution of the system from the definition in \cite{Pao}. Let us find $(\Bar{u},\Bar{v})$ such that $f(\Bar{u},\Bar{v})<0$ and $g(\Bar{u},\Bar{v})<0$ (the inequality can be non-strict). We can choose $\Bar{u} = \chi$, and $\Bar{v}$ a positive solution of the algebraic equation $\beta \chi \Bar{v}-\eta \Bar{v} -\delta \Bar{v}^2=0.$ Then, $$f(\Bar{u},\Bar{v}) \le 0 = \frac{\partial \Bar{u}}{\partial t} - \frac{\partial ^2\Bar{u}}{\partial x^2},\,\ \text{and}\,\, g(\Bar{u},\Bar{v}) \le 0 = \frac{\partial \Bar{v}}{\partial t} - d\frac{\partial ^2\Bar{v}}{\partial x^2},$$
the initial and boundary conditions are satisfied and $(\Bar{u},\Bar{v})$ is the upper solution of the system. Thus, Theorem 5.2 of \cite{Pao} shows that there exists a unique positive solution $(u(x,t),v(x,t))$ of the system (\ref{STmodel}-\ref{IC}) such that the following estimate hold $0 < u(x,t) < \Bar{u}, 0 < v(x,t) < \Bar{v}$ for $t\ge 0$ and $x \in \Omega.$ 
\end{proof}

\subsection{Local stability of steady state solutions}
A steady state $(u(x),v(x))$ of the system (\ref{STmodel})-(\ref{IC}) is the solution of the following system 
	\begin{equation}\label{eq:steady_state_system}
	\begin{aligned}
	   &\frac{\partial ^2u}{\partial x^2} + \nu u\Big(1-\frac{u}{\chi}\Big) -\frac{\beta uv}{1+\alpha u} =0,\\
	   &d\frac{\partial ^2v}{\partial x^2} + \varepsilon\Big(\frac{\beta u v}{1+\alpha u}- \eta v - \delta v^2\Big) =0,\\
	  & u_x(t,0)\,=\,u_x(t,L)\,=\,v_x(t,0)\,=\,v_x(t,L)\,=\,0,\,\,x \in \Omega,\, t\,>\,0.
	\end{aligned}
	\end{equation}
The feasible steady state solutions are $\Bar{E_0}=(0,0),\ \Bar{E_{\chi}}=(\chi,0)$ and $\Bar{E_*}=(u_*,v_*)$ where $(u_*,v_*)$ is the unique positive solution of $f(u_*,v_*)=0$ and $g(u_*,v_*)=0.$ We follow \cite{Camara11} to investigate the stability of the above mentioned steady state solutions. For that we linearize the system (\ref{STmodel}) around a steady state $\bar{E}$ and obtain the linearized system as follows
\begin{equation}\label{eq:spLinearized}
    \frac{\partial \textbf{W}}{\partial t} = D\Delta \textbf{W} + L(\bar{E})\textbf{W},\ \ \textbf{W}=(w_1,w_2)^T,
\end{equation}
where $$
    D=\begin{pmatrix}
    1&0\\0&d
    \end{pmatrix}
\ \text{and}\  L(u,v)=\begin{pmatrix}
\nu\Big(1-\frac{2u}{\chi}\Big)-\frac{\beta v}{(1+\alpha u)^2}&\frac{-\beta u}{1+\alpha u}\\\frac{\varepsilon \beta v}{(1+\alpha u)^2}&\varepsilon\Big(\frac{\beta u}{1+\alpha u}-\eta-2\delta v\Big)
\end{pmatrix}.$$

\begin{theorem} Let $\bar{E}_0$, $\bar{E}_{\chi}$ be the steady state solutions of system (\ref{eq:steady_state_system}) and $x\in \Omega.$ Then \\  
  (i) If $0<\varepsilon \ll 1$ and $\eta>0,\nu>0,$ then $\bar{E_0}$ is unstable.  \\
  (ii) If $\chi(\beta-\alpha)\ge \eta$ and $0<\varepsilon\ll 1,$ then $\bar{E}_{\chi}$ is unstable. Otherwise $\bar{E}_{\chi}$ is stable.
\end{theorem}
\begin{proof}
(i) The linearized system around $\bar{E_0}$ is given by
\begin{equation*}
    \begin{aligned}
      \frac{\partial w_1}{\partial t} &=\Delta w_1+\nu w_1,\\
      \frac{\partial w_2}{\partial t} &=d\Delta w_2-\varepsilon \eta w_2,\\
    \end{aligned}
\end{equation*}
with the above boundary condition. To show that $\bar{E_0}$ is unstable we need to prove that the largest eigenvalue of the following eigenvalue problem is positive
\begin{equation}\label{eq:eigenvalue_problem}
    \begin{aligned}
       \Delta w_1+\nu w_1 &= \sigma w_1,\\
        d\Delta w_2-\varepsilon \eta w_2 &= \sigma w_2,\\
    \end{aligned}
\end{equation}
with above zero-flux boundary condition. Let $\sigma_1$ be the largest eigenvalue of the above eigenvalue problem and $\lambda_p$ be the principle eigenvalue of the problem
\begin{equation*}
          \Delta w_1+\nu w_1 = \lambda w_1,\ \ x\in \Omega,\ \frac{\partial w_1}{\partial x}\Big |_{x=0,L}=0.
\end{equation*}
 Using the boundary condition we thus have $\lambda_p>0$ and the associated eigenfunction is $\Tilde{w}_1$. Now, if we take $w_2\equiv 0,$ then $(\Tilde{w}_1,0)$ satisfies the system (\ref{eq:eigenvalue_problem}) with $\sigma=\lambda_p.$ Thus, we constructed an eigenvalue $\lambda_p$ of the original system (\ref{eq:eigenvalue_problem}). Therefore we must have $\sigma_1\ge \lambda_p>0.$ Hence, $\bar{E_0}$ is unstable. \\ 

(ii) The linearized system around $\bar{E}_{\chi}$ is given as follows
\begin{equation*}
    \begin{aligned}
      \frac{\partial w_1}{\partial t} &=\Delta w_1-\nu w_1-\frac{\beta\chi}{\alpha \chi+1}w_2,\\
      \frac{\partial w_2}{\partial t} &=d\Delta w_2+\varepsilon \Big(\frac{\beta\chi}{\alpha \chi+1}-\eta\Big) w_2,\\
    \end{aligned}
\end{equation*}
with the above boundary condition. The corresponding eigenvalue problem is \begin{equation}\label{eq:eigenvalue_2}
\begin{aligned}
   \Delta w_1-\nu w_1-\frac{\beta\chi}{\alpha \chi+1}w_2 &= \sigma w_1,\\
    d\Delta w_2+\varepsilon \Big(\frac{\beta\chi}{\alpha \chi+1}-\eta\Big) w_2 &= \sigma w_2,
\end{aligned}
\end{equation}
where $x\in \Omega$ and satisfying zero-flux boundary condition. Let $\sigma_1$ be the largest eigenvalue of the above eigenvalue problem. Let $\chi(\beta-\alpha)>1$ and $\lambda_p$ be the principle eigenvalue of the problem
\begin{equation}\label{eq:eigenvalue_w2}
    d\Delta w_2+\varepsilon \Big(\frac{\beta\chi}{\alpha \chi+1}-\eta\Big) w_2 = \sigma w_2,\ \ x\in \Omega, \ \ \frac{\partial w_2}{\partial x}\Big|_{x=0,L}=0.
\end{equation}
Then by the previous argument $\lambda_p>0$ and the corresponding eigenfunction is $\Tilde{w_2}$. Also, $\Tilde{w_1}$ is the solution of the problem
\begin{equation*}
     \Delta w_1-(\nu +\sigma)w_1 = \frac{\beta\chi}{\alpha \chi+1}\Tilde {w_2},\ \ x\in \Omega , \ \frac{\partial w_1}{\partial x}\Big |_{x=0,L} =0. 
\end{equation*}
Then $(\Tilde{w_1},\Tilde{w_2})$ is the solution of eigenvalue problem (\ref{eq:eigenvalue_2}) with $\sigma=\lambda_p.$ Thus, the largest eigenvalue is positive and $\bar{E}_{\chi}$ is unstable.\\ 

Next, let $\chi(\beta-\alpha)<\eta$ and $(\Tilde{w_1},\Tilde{w_2})$ be the principle eigenfunction of the problem (\ref{eq:eigenvalue_2}) corresponding to the largest eigenvalue $\lambda_p.$ If $\Tilde{w_2}\not\equiv 0$, then $\lambda_p$ is also the eigenvalue of (\ref{eq:eigenvalue_w2}). Since $\chi(\beta-\alpha)<\eta$ we have $\lambda_p=\frac{\beta\chi}{\alpha \chi+1}-\eta<0.$ If $\Tilde{w_2} \equiv0,$ then $\Tilde{w_1} \not\equiv0$ and the largest eiqenvalue of the problem
\begin{equation*}
     \Delta w_1 -\nu w_1 = \sigma w_1,\ \ x\in\Omega, \ \ \frac{\partial w_1}{\partial x}\Big |_{x=0,L}=0,  
\end{equation*}
is $-\sigma<0.$ Hence $\bar{E}_{\chi}$ is stable.\\ \\

\end{proof}

\subsection{Turing instability}	

Let $\bar{E}_*=(u_*,v_*)$ denote the homogeneous steady-state of the system (\ref{eq:steady_state_system}). Linearizing the system in the vicinity of $\bar{E}_*$, we obtain the following system of linear PDEs. It describes the dynamics of initially small heterogeneous perturbations of $u$ and $v$ from their respective steady-states, in terms of the perturbation variables $w_1(t,x)$, $w_2(t,x)$:
\begin{equation}\label{LinSTeq}
\begin{aligned}
	  \frac{\partial w_1}{\partial t} &= \frac{\partial ^2w_1}{\partial x^2} + a_{11}w_1+a_{12}w_2,\\
	  \frac{\partial w_2}{\partial t} &= d\frac{\partial ^2w_2}{\partial x^2} + a_{21}w_1+a_{22}w_2.
	\end{aligned}
\end{equation}
We search the solution of (\ref{LinSTeq}) in the following form, 
\begin{equation}\label{eq:hetero_pertur}
    \begin{aligned}
    w_1(t,x) \, =\, \mu_1 e^{\lambda t}cos(kx),\,\,
    w_2(t,x) \, =\, \mu_2 e^{\lambda t}cos(kx),
  \end{aligned}
\end{equation}
satisfying the boundary condition (\ref{BC}) where $\mu_1,\ \mu_2\ll 1$ are two arbitrary constants, and $k$ is the wavenumber. Substituting (\ref{eq:hetero_pertur}) into the system (\ref{STmodel}) we obtain the characteristic equation 
\begin{equation}\label{cheqn}
|M-\lambda I_2|\,=\,0,
\end{equation}
where
\begin{equation}
M_k = \begin{pmatrix}
    a_{11}-k^2&a_{12}\\a_{21}&a_{22}-k^2d
    \end{pmatrix},
\end{equation}
and $a_{ij}$ are the same as in (\ref{jacobstar}).
The characteristic equation (\ref{cheqn}) can be written explicitly as follows
\begin{equation}\label{cheqn1}
   \lambda^2-((1+d)k^2-(a_{11}+a_{22}))\lambda + h(k^2)=0,
   \end{equation}
where tr$(M_k)=(1+d)k^2-(a_{11}+a_{22})$ and $\det(M_k)=h(k^2)$ and, 
\begin{equation}\label{eq:h_k}
    h(k^2) = k^4d-(a_{11}d+a_{22})k^2+a_{11}a_{22}-a_{12}a_{21}.
\end{equation}
The homogeneous steady-state is stable under small heterogeneous perturbations if both roots of the characteristic equation (\ref{cheqn1}) have negative real parts. Turing instability sets in if one root of the characteristic equation is zero at some critical wavenumber.
The critical wavenumber $k_{c}$ is obtained by solving $\frac{dh(k^2)}{d(k^2)}=0$ for $k^2$, which is given by
\begin{equation}
    k^2_{c} = \frac{1}{2d}(da_{11}+a_{22}).
\end{equation}
The Turing instability condition holds if $h(k^2)<0$ for certain feasible $k^2$ in the range of $(r^-,r^+)$ where 
$$r^-(d) = \frac{a_{11}d+a_{22}-\sqrt{(a_{11}d+a_{22})^2-4d(a_{11}a_{22}-a_{21}a_{12})}}{2d},$$
$$r^+(d) = \frac{a_{11}d+a_{22}+\sqrt{(a_{11}d+a_{22})^2-4d(a_{11}a_{22}-a_{21}a_{12})}}{2d}.$$
The critical wavenumber $k_c$ is a real number, $d>0$, feasible existence of $k_c$ demands the satisfaction of the implicit parametric restriction $da_{11}+a_{22}>0,$ and $a_{11}+a_{22}<0.$ These two condition can be simultaneously satisfied if $a_{11}$ and $a_{22}$ are of opposite sign. Equation of the Turing bifurcation curve can be obtained by substituting $k_c^2$ in $h(k^2)=0$ as follows
\begin{equation}\label{Turcurve}
   da_{11}+a_{22} \,=\,2\sqrt{d}\sqrt{a_{11}a_{22}-a_{12}a_{21}}.
\end{equation} 
The Turing bifurcation occurs if $a_{ij}$ satisfy the conditions $a_{11}a_{22}<0$ and $a_{11}a_{22}>a_{12}a_{21}$. These two conditions can be written in compact form as follows
\begin{equation}\label{stabcond}
    \textrm{Tr}(M_0)\,<\,0,\,\,\textrm{Det}(M_0)\,>\,0.
\end{equation} 
The wavenumber is given by $k=\frac{n\pi}{L}$, where $L$ is the length of the domain $[0,L]$, $n$ is a natural number. The expression for $k$ is guided by the no-flux boundary condition. Denote $\omega=\frac{\pi}{L}$, substituting $k=n\omega$ in the equation $h(k^2)=0$ and solving for $d$, we find the explicit expression for Turing bifurcation boundary corresponding to $n$-th mode,
\begin{equation}\label{eq:d_T(n)}
    \begin{aligned}
      d_T(n) \, = \, \frac{n^2\omega^2a_{22}-a_{11}a_{22}+a_{12}a_{21}}{n^2\omega^2(n^2\omega^2-a_{11})}.
    \end{aligned}
\end{equation}
The Turing bifurcation curve $\Gamma$ is given by the union of the boundaries of the curves $d_T(n)$ for all $n\ge 1$ (\cite{Mimura80}). Homogeneous steady-state $E_*$ is stable below the Turing curve $\Gamma$ and is unstable above it. Depending on the values of $n,\ \omega$, and $a_{ij},$ the value of $d_{T}(n)$ can either be positive or negative. We however, consider only the positive values of $d_T(n)$, which indicates the existence of multiple stationary heterogeneous solution for different values of $n\ (\ge 1)$. The pattern which we obtain through numerical simulation is determined by the possible values of $d_T(n)$ and the maximum value of real part $\lambda(k)$ which we will explain in details in the coming subsection.

\subsection{Turing patterns}

Here we explore the patterns produced by the model due to the Turing instability and their bifurcation through numerical examples. First we fix $\nu=10$, $\alpha=1$, $\beta=2.85$, $\eta=1$, $\varepsilon$, and $\delta$, $d$ are considered as the bifurcation parameters. We exploit the temporal dynamics of the model and consider two cases when system undergoes (a) super-critical Hopf bifurcation and (b)  sub-critical Hopf bifurcation. In both the cases we choose the parameter values such that the system has a unique coexistence state. The Turing bifurcation curve (black) along with the unstable modes are presented in Fig.~\ref{fig:self_diff_Turing}(a). We obtain infinitely many Turing bifurcation boundaries corresponding to the unstable modes in the pure Turing domain. But the question we want to address is, among all the unstable modes which $n^{th}$ mode participate in the pattern formation.

\begin{figure}[ht!]
		\centering
		\mbox{\subfigure[]{\includegraphics[scale=0.38]{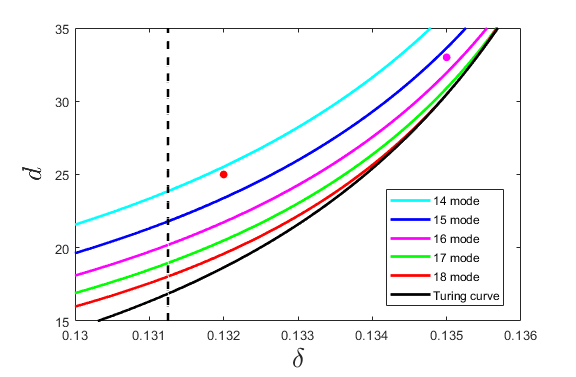}}
		\subfigure[]{\includegraphics[scale=0.34]{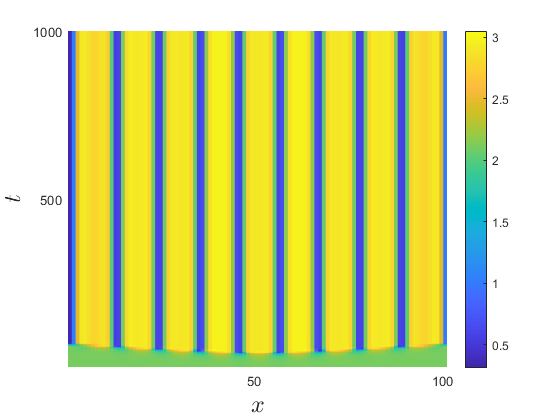}}
		\subfigure[]{\includegraphics[scale=0.34]{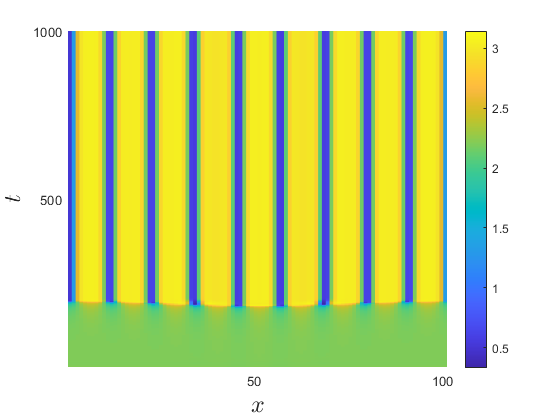}}}
		\caption{(a) The Turing instability curves for the unstable spatial modes are shown where the unstable spatial modes are marked in the legend and the dashed vertical Hopf line represents the Hopf curve. (b) The stationary Turing pattern showing approximately 9.5 peaks for parameter value in Turing domain  $(\delta,d)=(0.132,25)$, (c) stationary Turing pattern showing approximately 9 peaks for parameter value in Turing domain  $(\delta,d)=(0.135,33)$.
		Other parameter values are $\nu=10$, $\chi=6$, $\alpha=1$, $\beta=2.85$, $\eta=1$, $\varepsilon=1$.} 
	\label{fig:self_diff_Turing}
	\end{figure}	

We choose the parameter values $\delta=0.132$ and $d=25$, which lies below $d_T(14)$ but above the curves $d_T(15)$, $d_T(16)$, $d_T(17)$ and $d_T(18)$ as shown in Fig.~\ref{fig:self_diff_Turing}(a). It indicates that all the eigenmodes $n=15,16,17,18$ are unstable for the chosen parameter values. In order to understand the exact number of unstable eigenmodes and resulting stationary pattern we need to take help of the dispersion relation \cite{Murray,Manna,Mukherjee}. The largest real part of the eigenvalues $\lambda_k$ obtained from the dispersion relation is positive for $k_1^2=0.2015<k^2<0.6432=k_2^2$ and for $L=100$ we find the feasible values of $n$ within the range $14<n<25$. The most unstable eigenmode corresponds to $k^2_{max}=0.36$, $\textrm{Re}(\lambda_k)$ is maximum when $k=k_{max}$. Thus the wavelength for the mode that grows rapidly is $\frac{2\pi}{k_{max}} \approx 10.4719$. Hence the number of peaks we can expect for the stationary pattern within the spatial domain of size $L=100$ is $\frac{100}{10.4719} \approx 9.54$, which matches well with the numerical simulation result as shown in Fig.~\ref{fig:self_diff_Turing}(b). Next we consider another parameter set $\delta=0.135$, $d=33$,  for which the most unstable wavenumber is $k_{max}^2=0.3130$, and the wavelength for the most unstable mode is $\frac{2\pi}{k_{max}}=11.23$. Therefore, the number of peaks in the specified domain is $\frac{100}{11.23}\approx 8.9$ and we find the stationary pattern with nine peaks as shown in Fig.~\ref{fig:self_diff_Turing}(c). 

To understand the change of patterns due to the variation of parameter values and the involved bifurcations, we choose three different values of $\chi$ keeping $\delta=0.11$ fixed. The parameter value $\chi=3.8$ is inside the Turing instability domain but less than the supercritical Hopf-bifurcation threshold $\chi_{H_1}=3.966$. For this choice of $\chi$ we find stationary Turing pattern. Next we choose $\chi=4.5$ which belongs to Turing-Hopf domain and the resulting pattern for $d=20$ is homogeneous in space and oscillatory in time. Finally, we choose $\chi=12.25$ which is greater than the subcritical Hopf-bifurcation threshold $\chi_{H_2}=12.1$, and less than the saddle-node bifurcation threshold of limit cycles ($\chi_{SNL}=12.2522993$) for the temporal model and also inside the Turing instability domain. In this case we find stationary Turing pattern for $d=10>6.56=d_{cr}$ considering the initial condition as a small perturbation to the homogeneous steady state. However, for a large homogeneous perturbation around the steady state we find homogeneous in space and oscillatory in time solution. This is due to the presence of the unstable limit cycle, which acts as separatrix for Turing and non Turing solutions. Furthermore, due to the coexistence of homogeneous steady state and stable limit cycle in the Turing domain, these two solution branches interact near the separatrix giving rise to either the homogeneous in space and oscillatory in time solution or the regular/irregular oscillatory pattern (see, respectively, Figs.~\ref{Fig:transient_eps1} and Fig.~\ref{fig:trans_chi11dot9}(d) below). Interestingly, two different solution type can coexist in space either for a considerable time or even indefinitely (cf.~Fig.~\ref{fig:trans_chi11dot9}(h)). 

Now we consider the spatio-temporal pattern formation in the presence of slow-fast time scale where $0<\varepsilon \le 1$. In the previous section we have discussed that the temporal Hopf-bifurcation threshold alters with the variation of $\varepsilon$ and the position of Turing bifurcation curve changes while $\varepsilon$ decrease from 1. The temporal Hopf-bifurcation threshold increases and the pure Turing domain shifts upwards in the $\delta-d$ parameter space as shown in Fig.~\ref{fig:Turing_curve_eps}(a). Other parameter values are fixed at $\nu=10$, $\chi=6$, $\alpha=1$, $\beta=2.85$ and $\eta=1$. This implies that with decreasing $\varepsilon,$ the stability region  for homogeneous steady-state in the parametric space decreases and also we obtain a very narrow Turing domain even for large value of $d$. Therefore, the stationary Turing solution obtained for $\varepsilon=0.1$ loses its stability and oscillatory in time solution is obtained for sufficiently small values of $\varepsilon.$ Keeping $d=10$ and $\varepsilon=0.01$ fixed, the homogeneous in space and oscillatory in time solution of the spatio-temporal model bifurcates from the stationary steady state at $\delta=0.14444.$ The small amplitude periodic solution exists nearly upto $\delta=0.14442.$ Further decreasing $\delta,$ we observe a sharp rise in the average density of the prey species in an extremely narrow interval of $\delta$ exhibiting spatio-temporal canard explosion \cite{Avitabile17}. The bifurcation diagram explaining this phenomenon is shown in Fig.~\ref{fig:spatial_canard_exlosion}. The phase-space trajectory of the spatial average of the prey and predator density coincides with the phase-space trajectory of the non-spatial case (cf. Fig.~\ref{fig:canard_super-critical}). Further, for $\varepsilon \rightarrow 0,$ the solution of the spatially extended system resembles like relaxation oscillation. 

\begin{figure}
    \centering
    \includegraphics[scale=0.55]{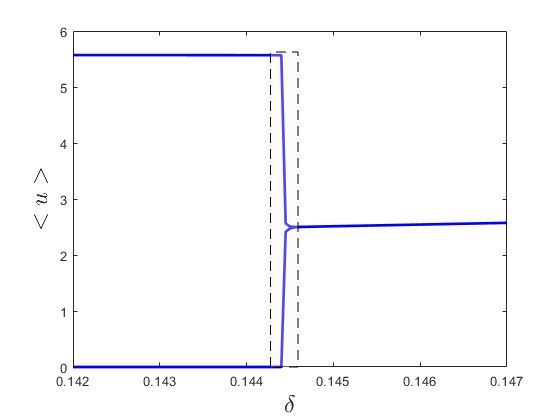}
    \caption{The plot of the spatial average of the prey density against parameter $\delta,$ exhibiting spatio-temporal canard explosion in a narrow interval for $\varepsilon=0.01$. Other parameter values are fixed at $\nu=10,\ \chi=6,\ \alpha=1,\ \beta=2.85, \eta=1,\ d=10.$}
    \label{fig:spatial_canard_exlosion}
\end{figure}

The Hopf-bifurcation curve and Turing bifurcation curves with different eigenmodes are shown in Fig.~\ref{fig:Turing_curve_eps}(b) when $\varepsilon=0.1$. The instability curves for the unstable eigenmodes $n=6,7,8,9$ are presented which form the boundary of the Turing instability domain. We can verify analytically as before, that for $\delta=0.1432$, $d=140$, (marked with magenta dot in \ref{fig:Turing_curve_eps}(b)) the rapidly growing eigenmode is 7, whereas for $\delta=0.1437$, $d=340$, (blue dot in \ref{fig:Turing_curve_eps}(b)) eigenmode 6 is the most unstable and we find stationary Turing pattern. When $\varepsilon\ll 1$, the size of the temporal unstable limit cycle shrink in size  which makes it difficult to establish stationary Turing pattern even in pure Turing domain. For parameter values very close to the co-dimension two Turing-Hopf point, the amplitude of the pattern is restricted by the small unstable cycle. 

\begin{figure}[ht!] 
		\centering
		\mbox{\subfigure[]{\includegraphics[width=7cm, height=5cm]{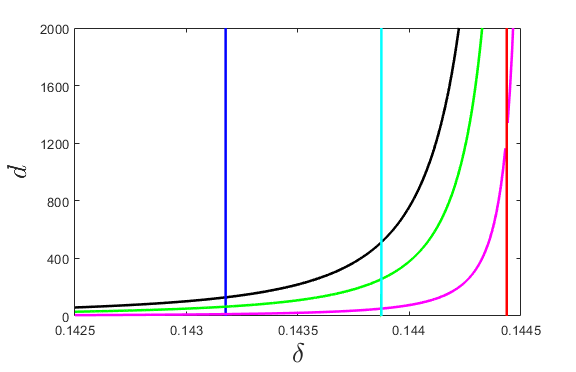}}	    \subfigure[]{\includegraphics[width=7cm, height=5cm]{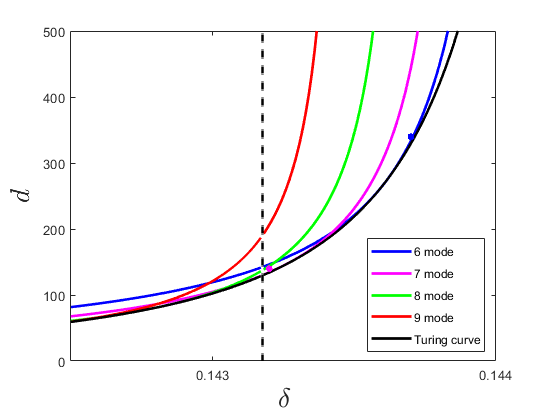}}}
		\caption{(a) The Turing (black, green, magenta) and Hopf curves (blue, orange, red) for $\varepsilon=0.1$, $\varepsilon=0.05$,  and $\varepsilon=0.01$, (b) The Turing curve (black) for $\varepsilon=0.1$ along with few unstable modes are shown. Other parameter values are $\nu=10$, $\chi=6$, $\alpha=1$, $\beta=2.85$, $\eta=1$.} \label{fig:Turing_curve_eps}
\end{figure}

\section{Transient spatio-temporal dynamics}

In this section, we study the nature of the transient dynamics observed in the spatio-temporal model for $\varepsilon\le 1$. We consider three cases where intriguing transient dynamics are obtained:

\medskip 

\noindent (a) The coexistence steady state $E_*$ is the only attractor of the temporal model;

\smallskip

\noindent (b) The temporal model exhibits bistability, that is, a stable coexistence steady state exists along with a stable limit cycle and their basin of attraction is separated by a separtrix which is an unstable limit cycle;

\smallskip

\noindent (c) The coexistence steady stable is unstable and the stable limit cycle surrounding $E_*$ is the attractor of the system.

\medskip 

We simulate our model choosing different parameter values corresponding to the cases mentioned above with spatial domain $L=200$ with suitable choices of $\Delta t$ and $\Delta x.$ We consider the following initial condition to identify different transient behaviour in the above mentioned cases 
\begin{eqnarray}\label{IC1}
u(x,0)\,=\,\left\{
\begin{array}{ll}
u_*+0.02, & |x-100|\,<\,2\\
u_*,  & |x-100|\,\geq\,2 \\
\end{array}\right.,\,\,\,\,\,v(x,0)\,=\,\left\{
\begin{array}{ll}
v_*+0.01, & |x-100|\,<\,2\\
v_*,  & |x-100|\,\geq\,2 \\
\end{array}\right.,
\end{eqnarray}

Since our main focus is on the parameter range where the system exhibits self-organized pattern formation, for a better understanding of the trajectories we introduce two auxiliary variables to quantify the degree of the population distribution's spatial heterogeneity. Namely, we consider  $$u_{ampl}(t)=u_{max}(t)-u_{min}(t)$$ and $$u_{grad}(t) = \sqrt{\int_{0}^{L} \Big(\frac{du}{dx}\Big)^2dx}$$ such that these two variables give the norm of the variable $u(x,t)$ in functional spaces $C^0$ and $L^2$ respectively. Similarly, we define these norms for the function $v(x,t)$. 
The spatial average of $u$ is denoted as $\langle u\rangle$ and is defined by $\langle u\rangle\,=\,\frac{1}{L}\int_0^Lu(y,t)dy$, with a similar definition for $\langle v\rangle$.\\

\begin{figure}[ht!] 
	\centering
	\mbox{\subfigure[]{\includegraphics[scale=0.35]{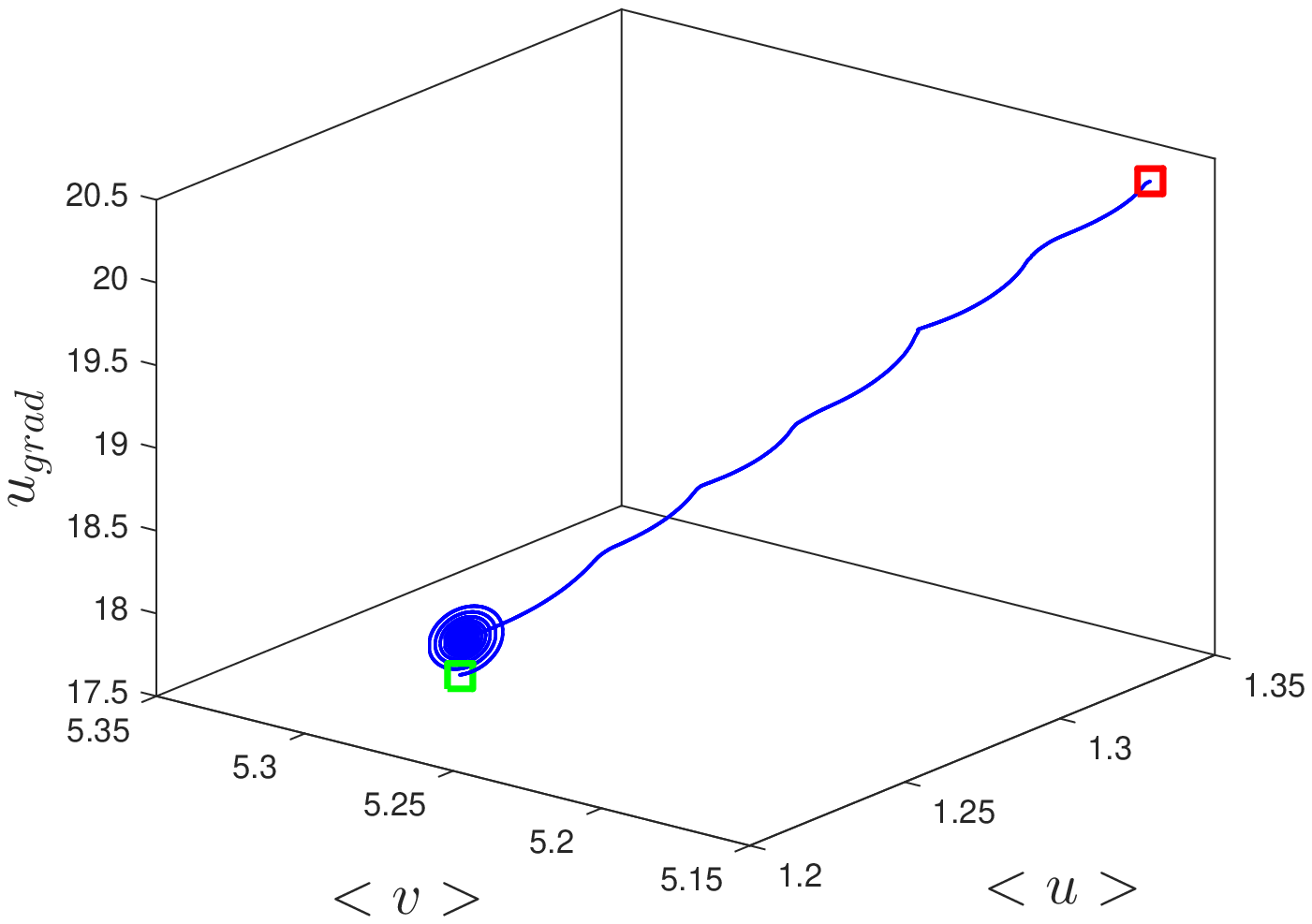}}
		\subfigure[]{\includegraphics[scale=0.35]{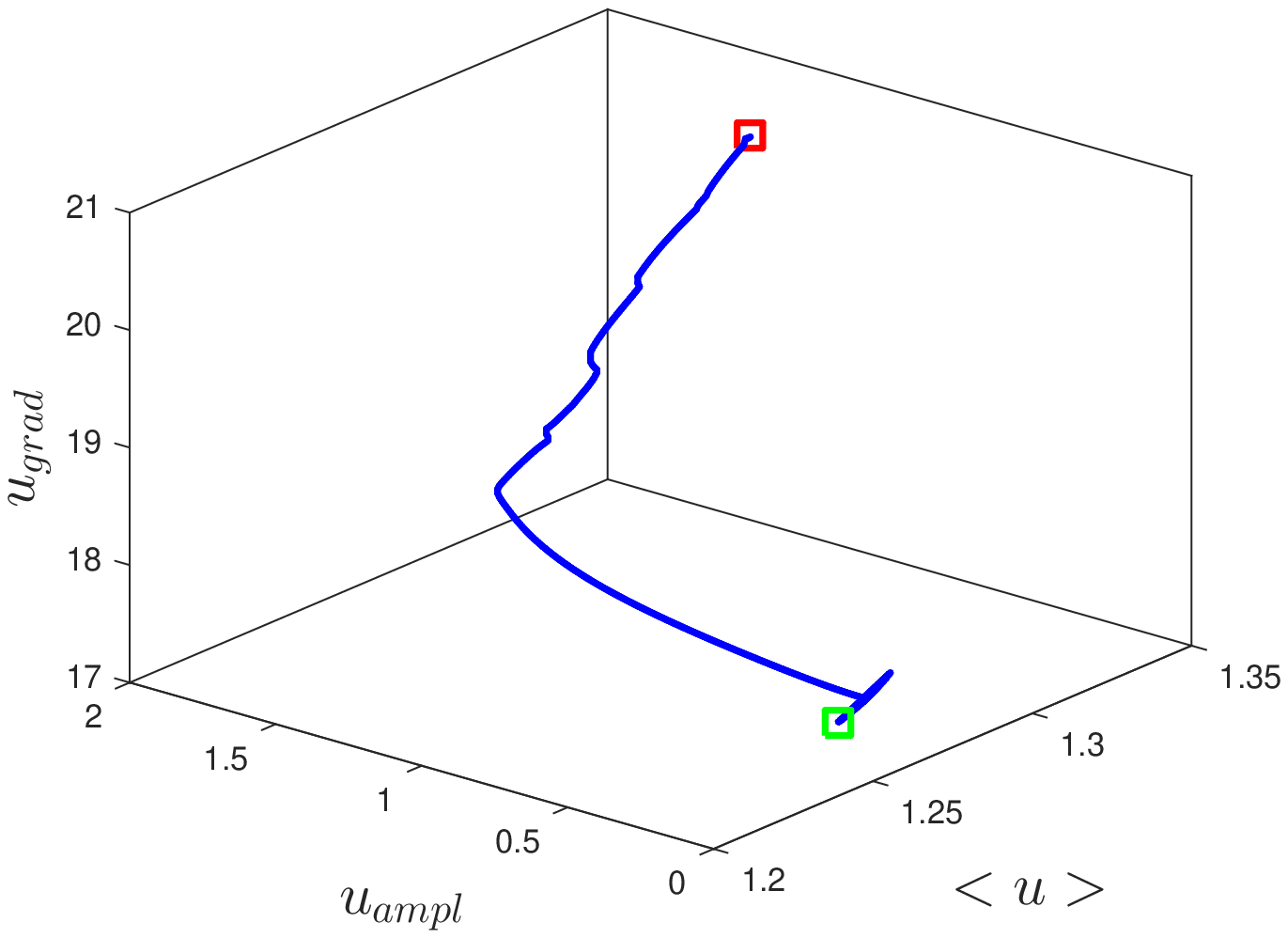}}}
		\mbox{\subfigure[]{\includegraphics[scale=0.35]{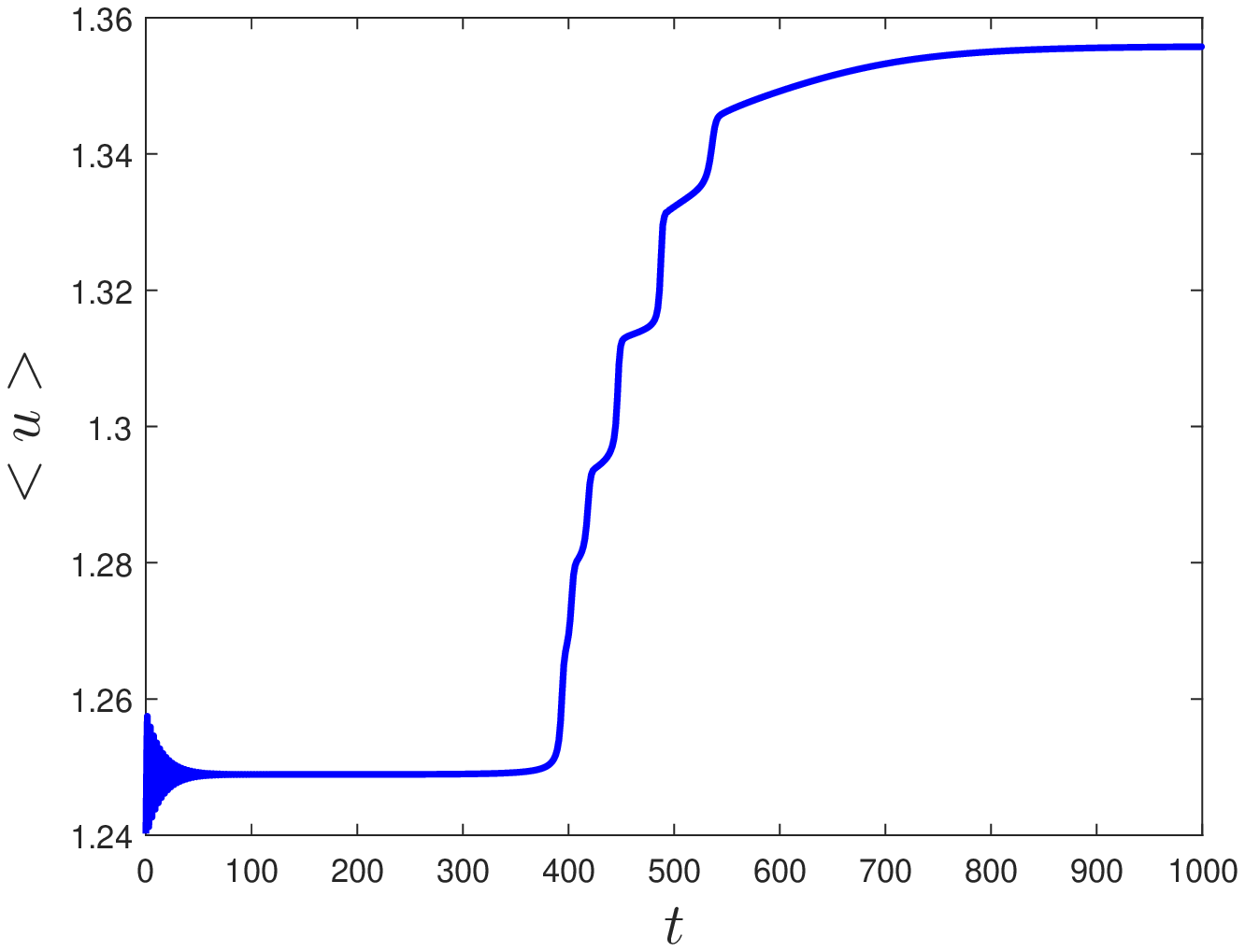}}
		\subfigure[]{\includegraphics[scale=0.35]{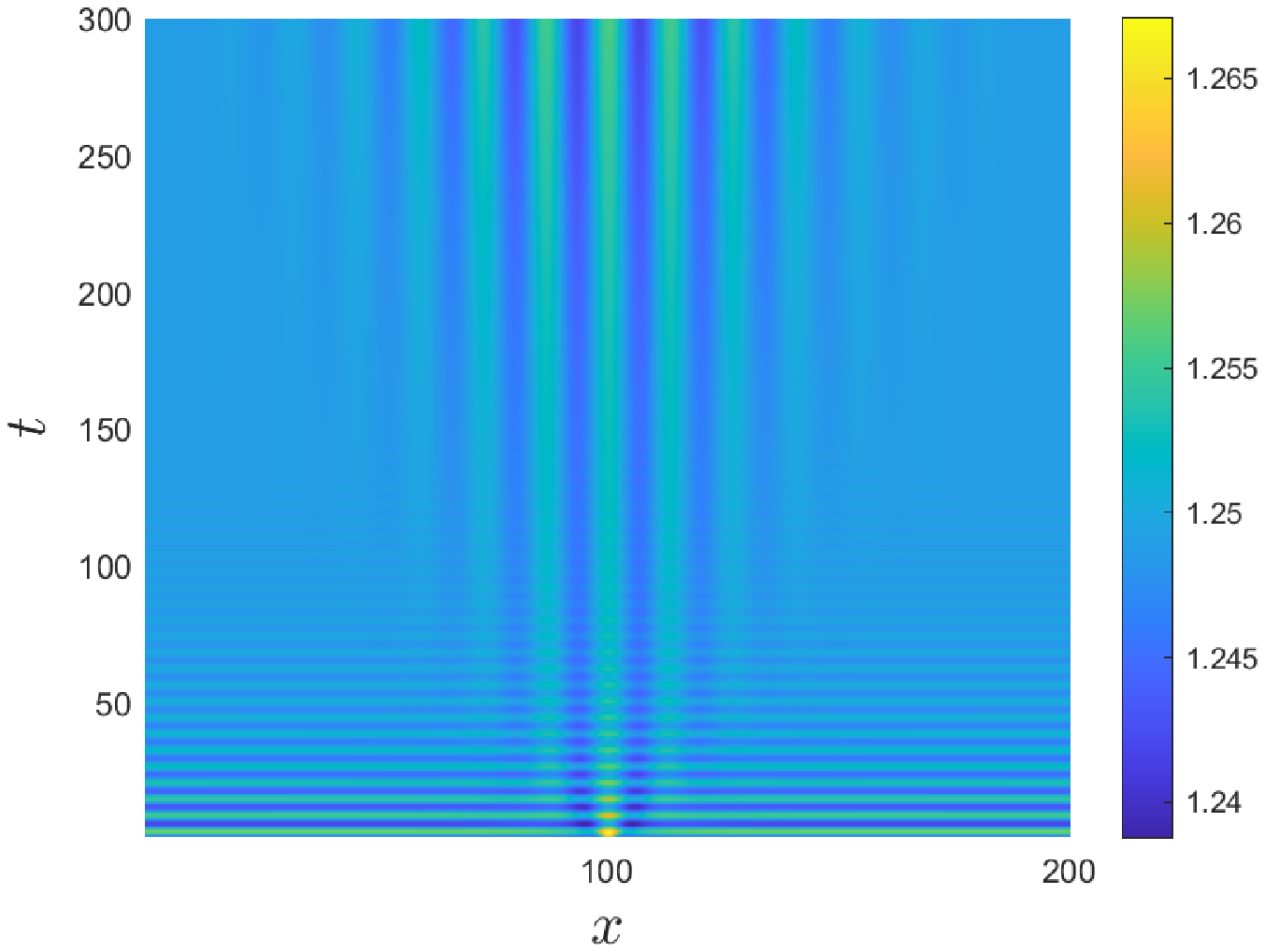}}}
	\caption{Spatio-temporal dynamics of system (\ref{STmodel}) quantified in different ways: (a) the plot of $(\langle u\rangle, \langle v \rangle,u_{grad})$, (b) the trajectory of $(\langle u \rangle,u_{ampl},u_{grad}),$ (c) the initial transient dynamics, i.e plot of $(\langle u\rangle,t)$ and (d) the transient patterns. (a,b,c,d) for $\chi=3.8,\ d=100$. Other parameters are $\nu=10$, $\alpha=1$, $\beta=2.85$, $\eta=1$, $\varepsilon=1$ and $\delta=0.11.$ Green and red square mark the initial and end points respectively.
	} \label{fig:sptemp_chi3dot8}
\end{figure}

We first study the case for $\varepsilon=1,$ and then further we will show how the transient dynamics changes with decreasing $\varepsilon$. We fix $\delta=0.11$. Then for $\chi=3.8$, the system has a unique stable steady state and the critical value of diffusion for Turing instability is $d_{cr}=94.26.$  Thus, for $\ d=100$ $(d>d_{cr})$, the parameter value lies in the pure Turing domain and the transient obtained is presented in Fig.~\ref{fig:sptemp_chi3dot8}. Fig.~\ref{fig:sptemp_chi3dot8}(a) shows the plot of $(\langle u\rangle,\langle v\rangle,u_{grad})$, Fig.~\ref{fig:sptemp_chi3dot8}(b) shows the trajectory of $(\langle u\rangle,u_{ampl},u_{grad}),$ Fig.~\ref{fig:sptemp_chi3dot8}(c) shows the variation of the spatial average of prey density with respect to time and Fig.~\ref{fig:sptemp_chi3dot8}(d) shows the corresponding transient pattern. From the time series plot, we see that after some initial oscillations, the system stays near the steady state for a considerable time before settling to a different stationary state with higher norm value. For the temporal model, at $\chi=12.25$ we observe bi-stability, that is, a stable limit cycle and a stable steady state coexists. For the spatio-temporal model, depending on the value of $d$ and the initial state of the system, the trajectory either converges to a stationary steady state after some initial oscillatory transient or to homogeneous in space and periodic in time solution. Because of the presence of an unstable limit cycle in the temporal model, the transient dynamics is chaotic. This is presented in Fig.~\ref{fig:sptemp_chi12dot25}. Note that, depending on the magnitude of parameter $d$, the nature of the oscillation remains the same (chaotic) but the duration of the transient increases to infinity when $d\rightarrow d_{cr} (\approx 6.56)$.  Therefore, whenever the background parameters of the system is close to the bifurcation threshold, it exhibits long chaotic transient. This is illustrated in Fig.~\ref{fig:transient_consolidated} for different values of $d$ ranging from $d=6.6$ to $d=15$. The rate of increase of the duration of the transient chaotic regime when $d$ approaches $d_{cr}$ is well described by a power law (which transforms into a straight line in the log-log scale), as is readily seen from the graph shown in Fig.~\ref{fig:transient_consolidated}c.

For $\chi=4$ and $d=50$, the steady state is unstable and the spatio-temporal limit cycle almost coincides with the temporal limit cycle. However, using the same initial condition, for $d=55$, the system dynamics converges to a stationary state; see Fig.~\ref{fig:trans_chi4}. Therefore, changing the value of $d$ but keeping other parameters fixed, either a stationary solution or periodic regime occurs. For $\chi=11.9,$ the coexistence equilibrium is unstable surrounded by a stable limit cycle but very near to subcritical Hopf threshold. Here the system approaches a periodic orbit after spending a significant amount of time enclosing a surface in the three dimensional space, as if there exists an attractor of the system, hence forming a ghost attractor. This gives rise to long transient. Two different types of transient dynamics obtained for $d=5,$ and $d=25$ are shown in Fig.~\ref{fig:trans_chi11dot9}.

\begin{figure}[ht!]
	\centering
		\mbox{\subfigure[]{\includegraphics[scale=0.29]{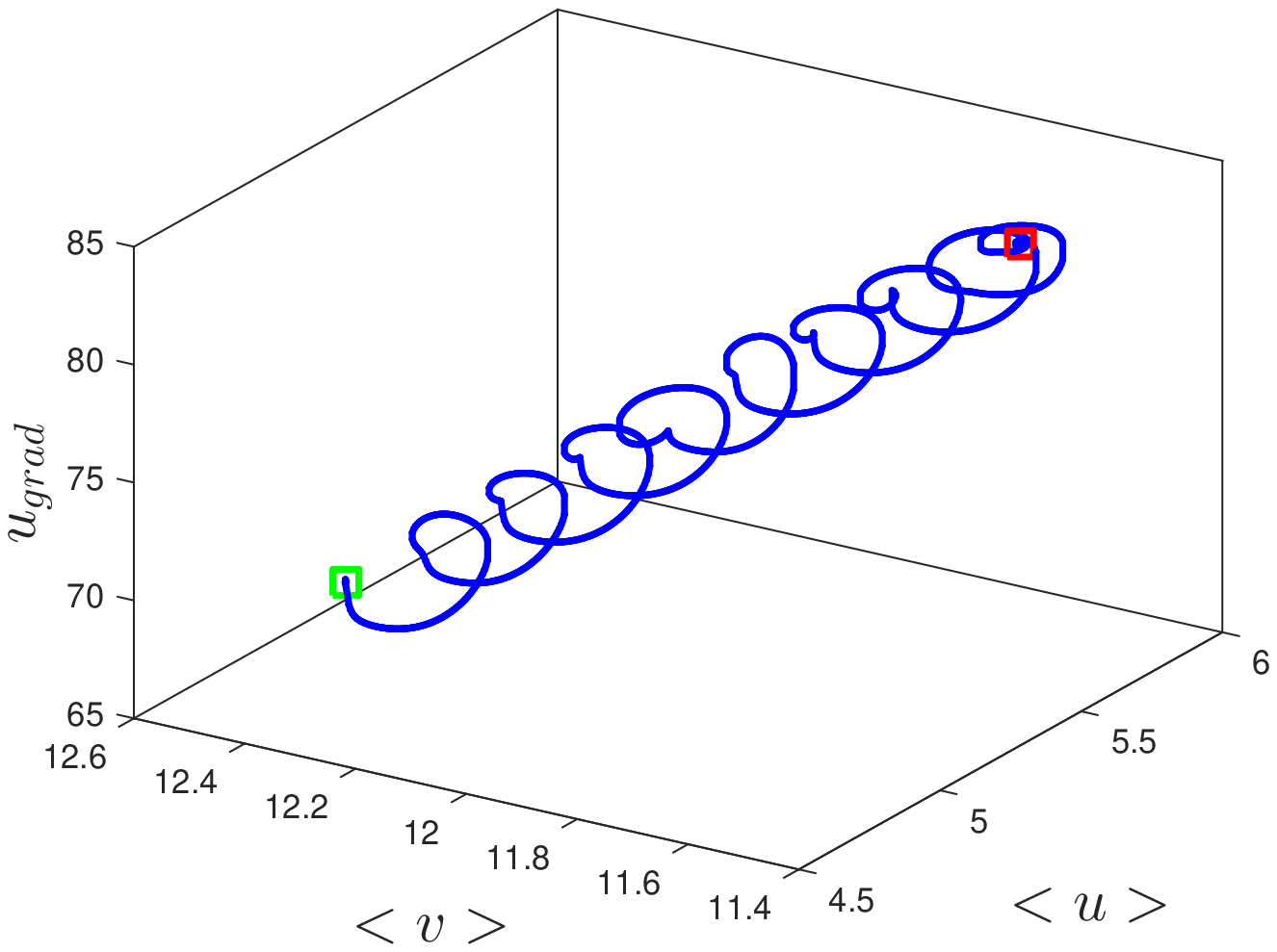}}
		\subfigure[]{\includegraphics[scale=0.29]{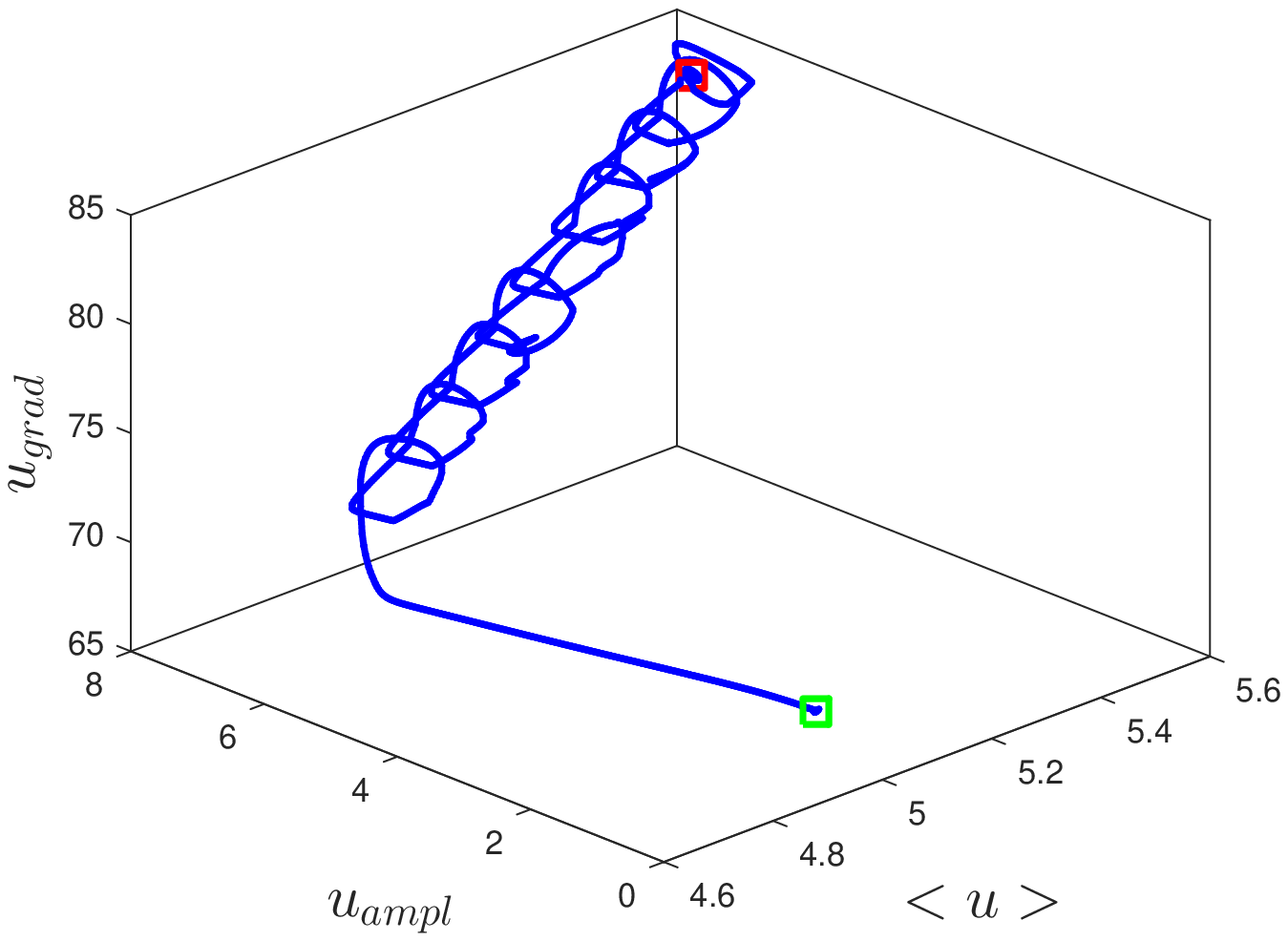}}}
	\mbox{\subfigure[]{\includegraphics[scale=0.24]{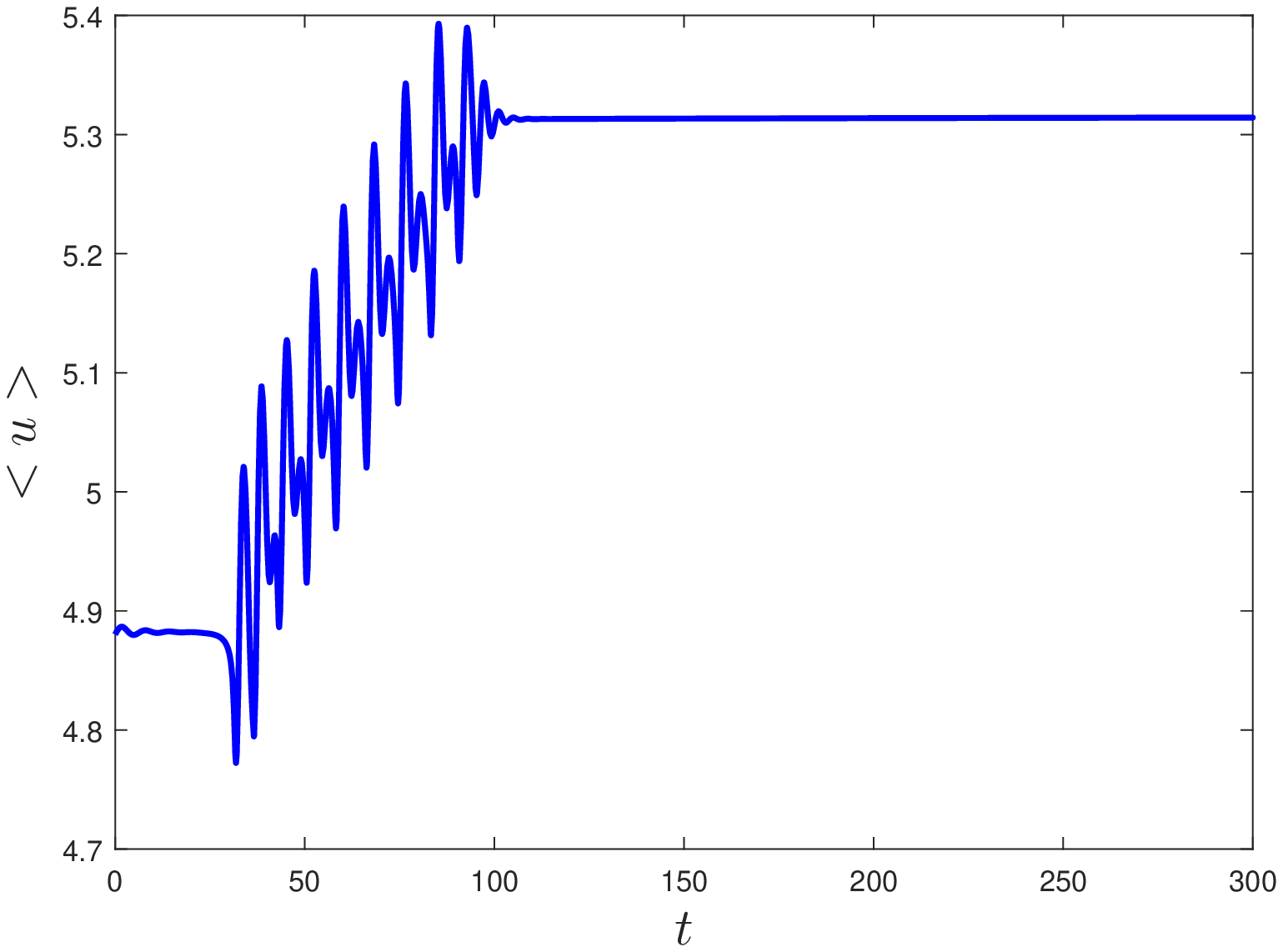}}
		\subfigure[]{\includegraphics[scale=0.28]{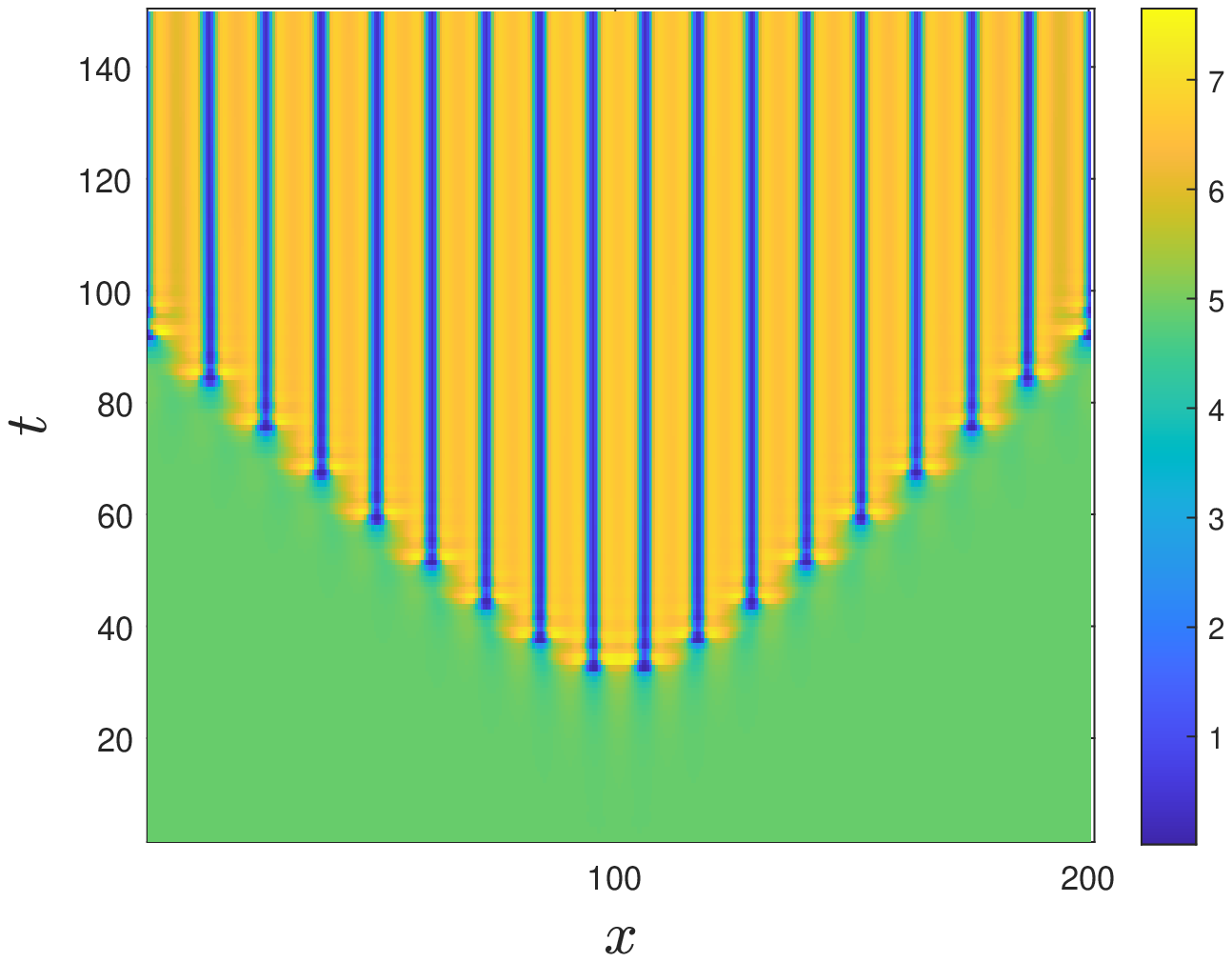}}}
	\caption{Spatio-temporal dynamics of system (\ref{STmodel}) quantified in different ways: (a) the plot of $(\langle u\rangle,\langle v\rangle,u_{grad})$, (b) the trajectory of $(\langle u\rangle,u_{ampl},u_{grad}),$ (c) the initial transient dynamics, i.e plot of $(\langle u\rangle,t)$, and (d) the transient patterns for $\chi=12.25,\ d=10$. Other parameters are $\nu=10$, $\alpha=1$, $\beta=2.85$, $\eta=1$, $\varepsilon=1$ and $\delta=0.11.$. Green and red square mark the initial and end points respectively.} \label{fig:sptemp_chi12dot25}
\end{figure}

\begin{figure}[ht!]
	\centering
		\mbox{\subfigure[]{\includegraphics[scale=0.33]{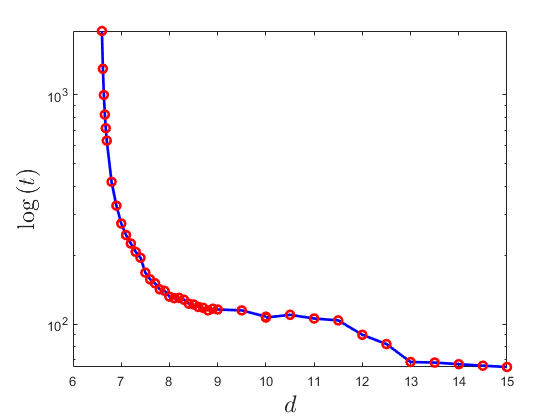}}
		\subfigure[]{\includegraphics[scale=0.33]{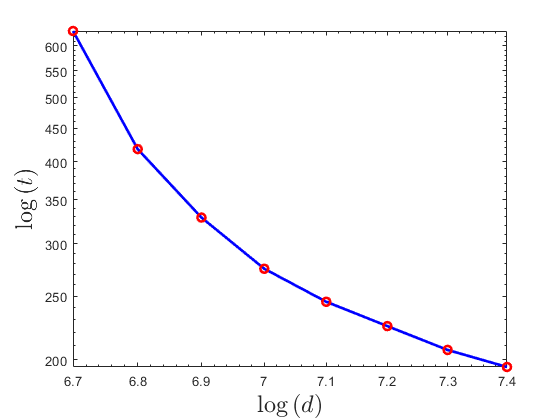}}
		\subfigure[]{\includegraphics[scale=0.33]{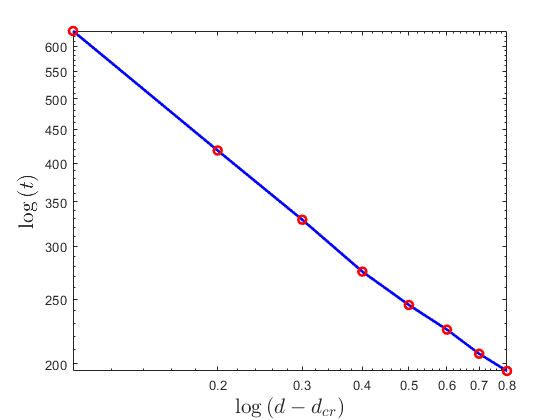}}
		}
	  \caption{(a) The duration of the transient is plotted against the value of $d \in [6.7,15].$ (b) Zoomed figure of (a) for $d \in [6.7, 7.4]$. (c) The same as in (b) but shown in log-log scale (logarithmic for both axes). The plot is very close to a straight line, which indicates that the transient duration depends on $|d-d_{cr}|$ as a power law.} \label{fig:transient_consolidated}
\end{figure}

\begin{figure}[ht!]
	\centering
	\mbox{\subfigure[]{\includegraphics[width=4cm,height=2.5cm]{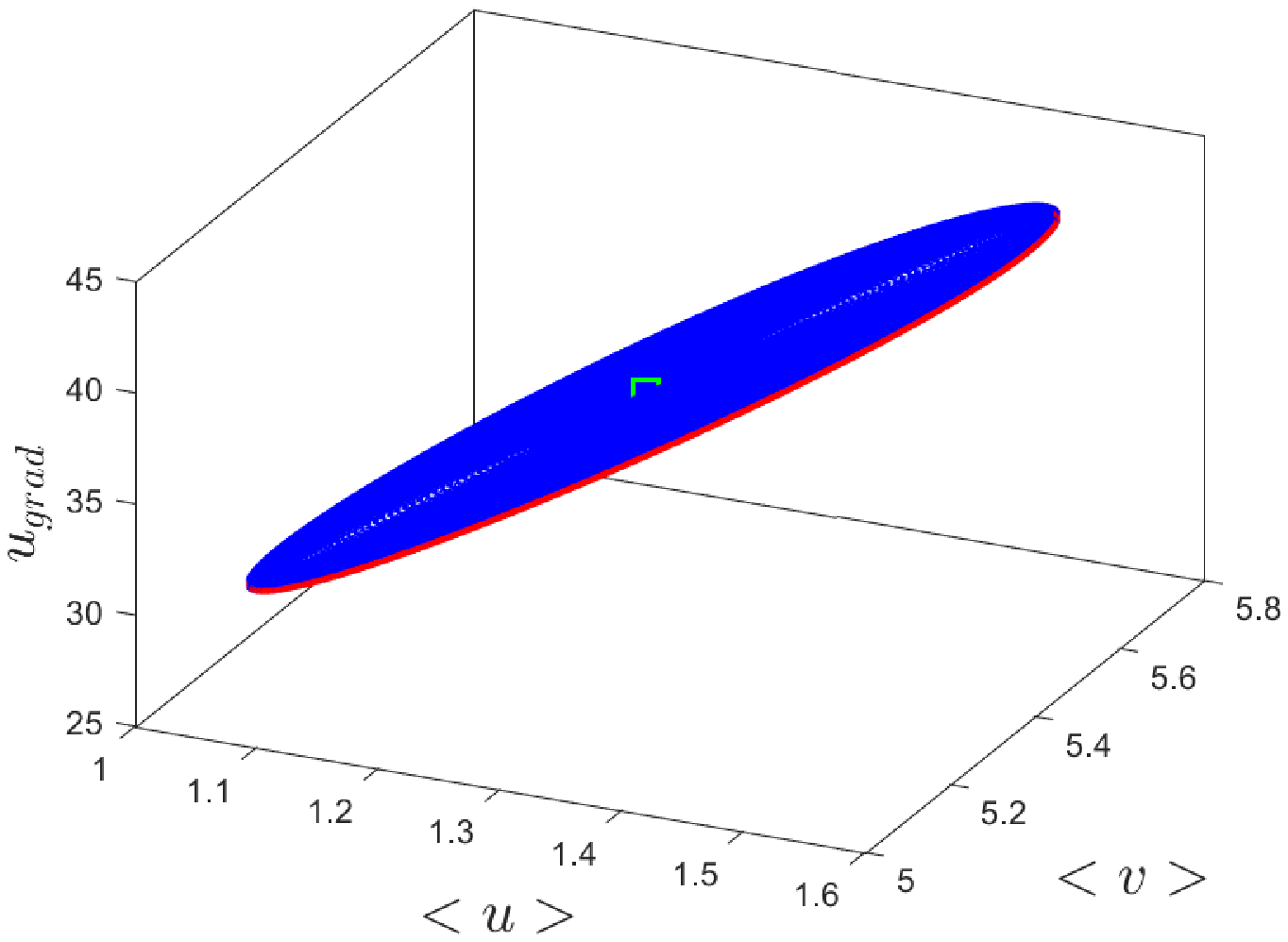}}
		\subfigure[]{\includegraphics[width=4cm,height=2.5cm]{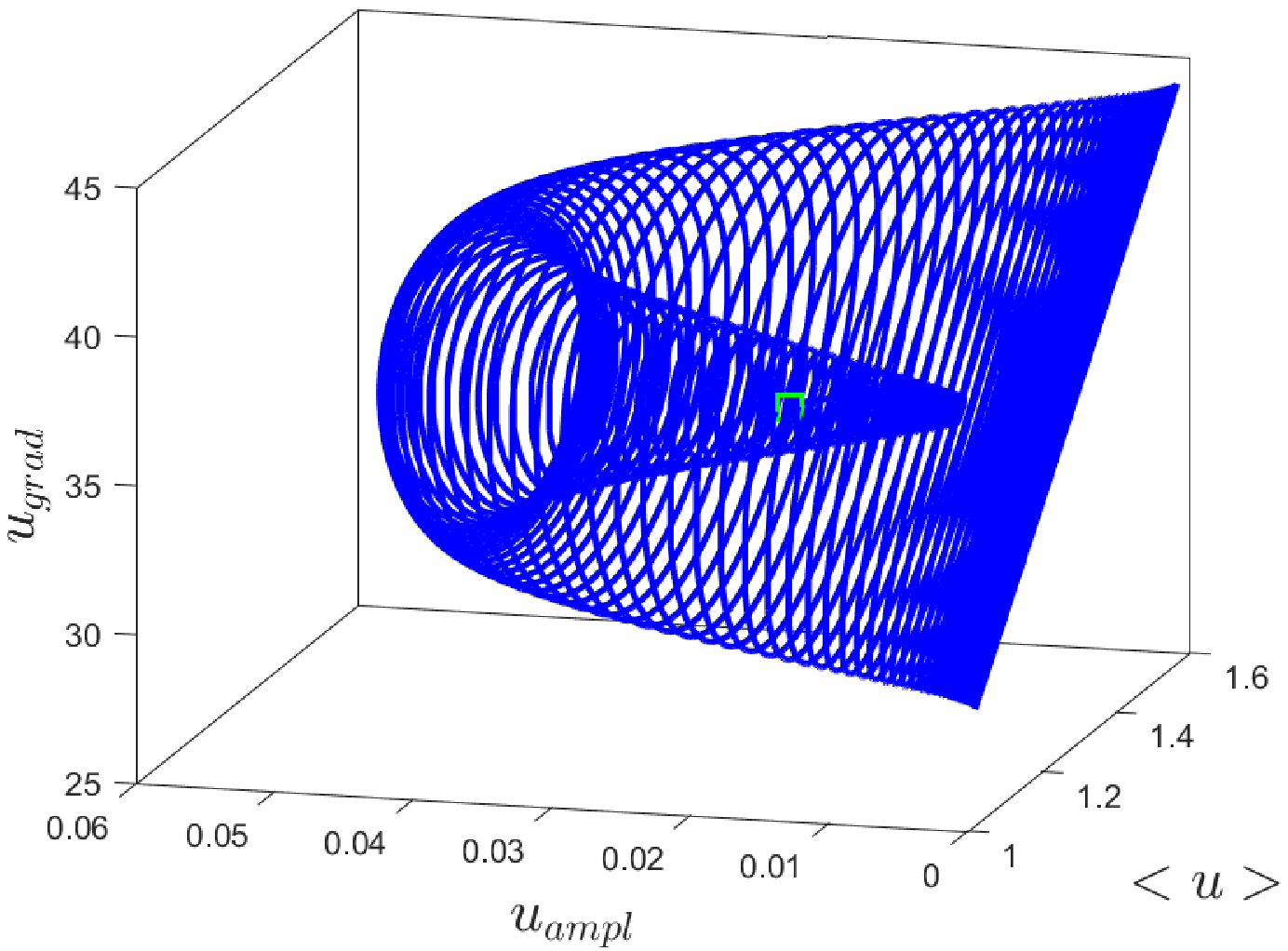}}}
		\mbox{\subfigure[]{\includegraphics[width=4cm,height=2.5cm]{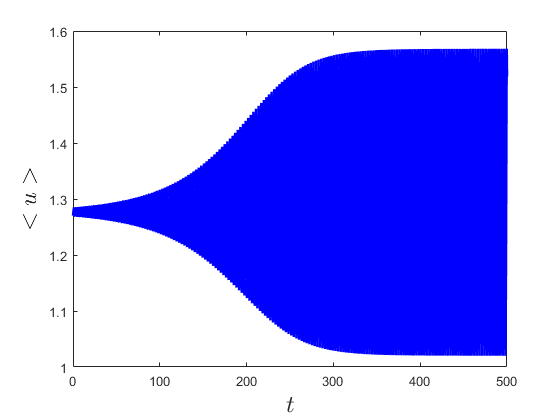}}
       \subfigure[]{\includegraphics[width=4cm,height=2.5cm]{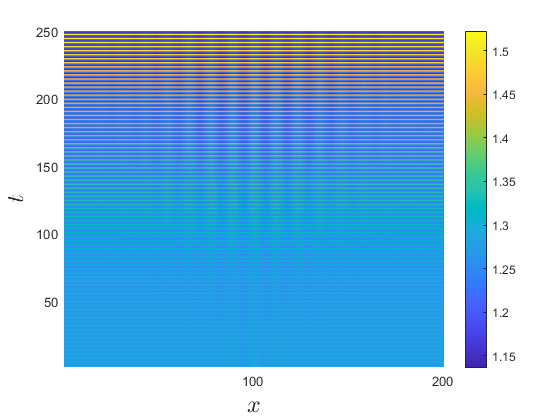}}}
	\mbox{\subfigure[]{\includegraphics[width=4cm,height=2.5cm]{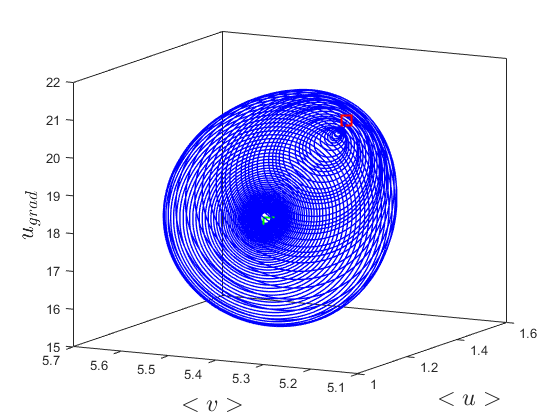}}
		\subfigure[]{\includegraphics[width=4cm,height=2.5cm]{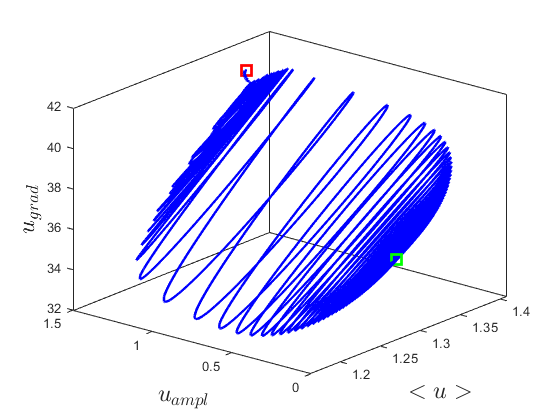}}}
		\mbox{\subfigure[]{\includegraphics[width=4cm,height=2.5cm]{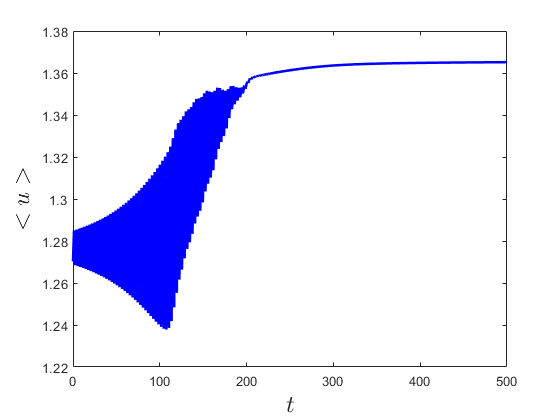}}
		\subfigure[]{\includegraphics[width=4cm,height=2.5cm]{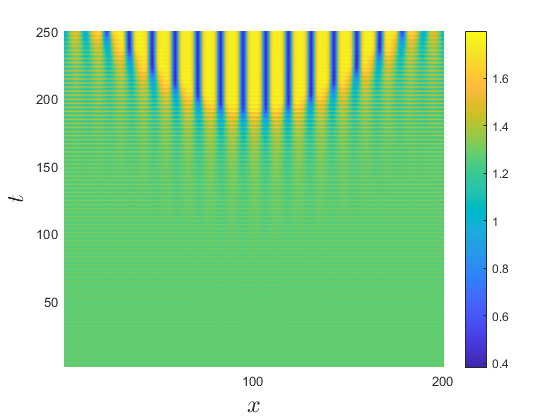}}}
		\caption{Spatio-temporal dynamics of system (\ref{STmodel}) quantified in different ways for $\chi=4,\ \delta=0.11,$ and $d=50$ (first and second panels), $d=55$ (third and fourth panels): (a) the plot of $(\langle u\rangle,\langle v\rangle,u_{grad})$, (b) the trajectory of $(\langle u\rangle,u_{ampl},u_{grad}),$ (c) the initial transient dynamics, i.e plot of $(\langle u\rangle,t)$, and (d) the transient patterns. Green and red square mark the starting and ending points respectively. Other parameter values are $\nu=10$, $\alpha=1$, $\beta=2.85$, $\eta=1$, $\varepsilon=1$.}
	\label{fig:trans_chi4}
\end{figure}

Now, while decreasing $\varepsilon$, the global attractor of the system changes from stable steady state to periodic attractor depending on the other temporal parameters. Though the initial transient time decreases considerably when limit cycle is the attractor of the system, but there exists long period of stasis and rapid oscillations.  For instance, in Fig.~\ref{Fig:transient_eps1}(a) the transient time and the irregular periodic oscillations increases with decreasing $\varepsilon$, whereas in Fig.~\ref{Fig:transient_eps1}(b) the initial transient decreases giving rise to large amplitude spatio-temporal canard like solution which is homogeneous in space and oscillatory in time, see Fig.~\ref{Fig:transient_eps1}(c). We validate this by tracking the variables $u$ and $v$ at a particular space point. From the time series analysis in Fig.~\ref{Fig:transient_eps1}(c) we observe that the slow variable $u$ takes much longer time in the transition from high density to low, which is along the critical manifold $C_0^1$. After some initial transients the slow flow of the trajectory along the critical manifold makes a fast jump from the vicinity of the fold point. We therefore conclude the following points:
\begin{itemize}
	\item If a stable equilibrium point is the only attractor of the temporal model, then with decreasing $\varepsilon$ the transient increases in the corresponding spatio-temporal model and it takes longer time to settle down to a stationary heterogeneous state (cf. Fig.~\ref{Fig:transient_eps1}(a)).
	
	\item When the system is bistable, and specifically for parameter values very near to the saddle-node bifurcation threshold of limit cycles ($\delta=0.11,\ \chi_{SNL}=12.2523$), we initially observed long chaotic transient. But for $\varepsilon<1$ the transient time and the irregularity of the transient decreases. But there exists alternate period of stasis and sudden jumps (cf. Fig.~\ref{Fig:transient_eps1}(b)).
		\item If a stable limit cycle, surrounding an unstable coexistence steady-state, is the only attractor of the temporal system, then with decreasing $\varepsilon$ the transient decreases. The number of oscillations in $t\in[0,400]$ decreases with $\varepsilon\ll 1.$ However, the time taken for one complete cycle increases with decreasing $\varepsilon.$ (cf. Fig.~\ref{Fig:transient_eps1}(c)). 
\end{itemize}

\begin{figure}[ht!]
	\centering
	\mbox{\subfigure[$d=5$]{\includegraphics[scale=0.3]{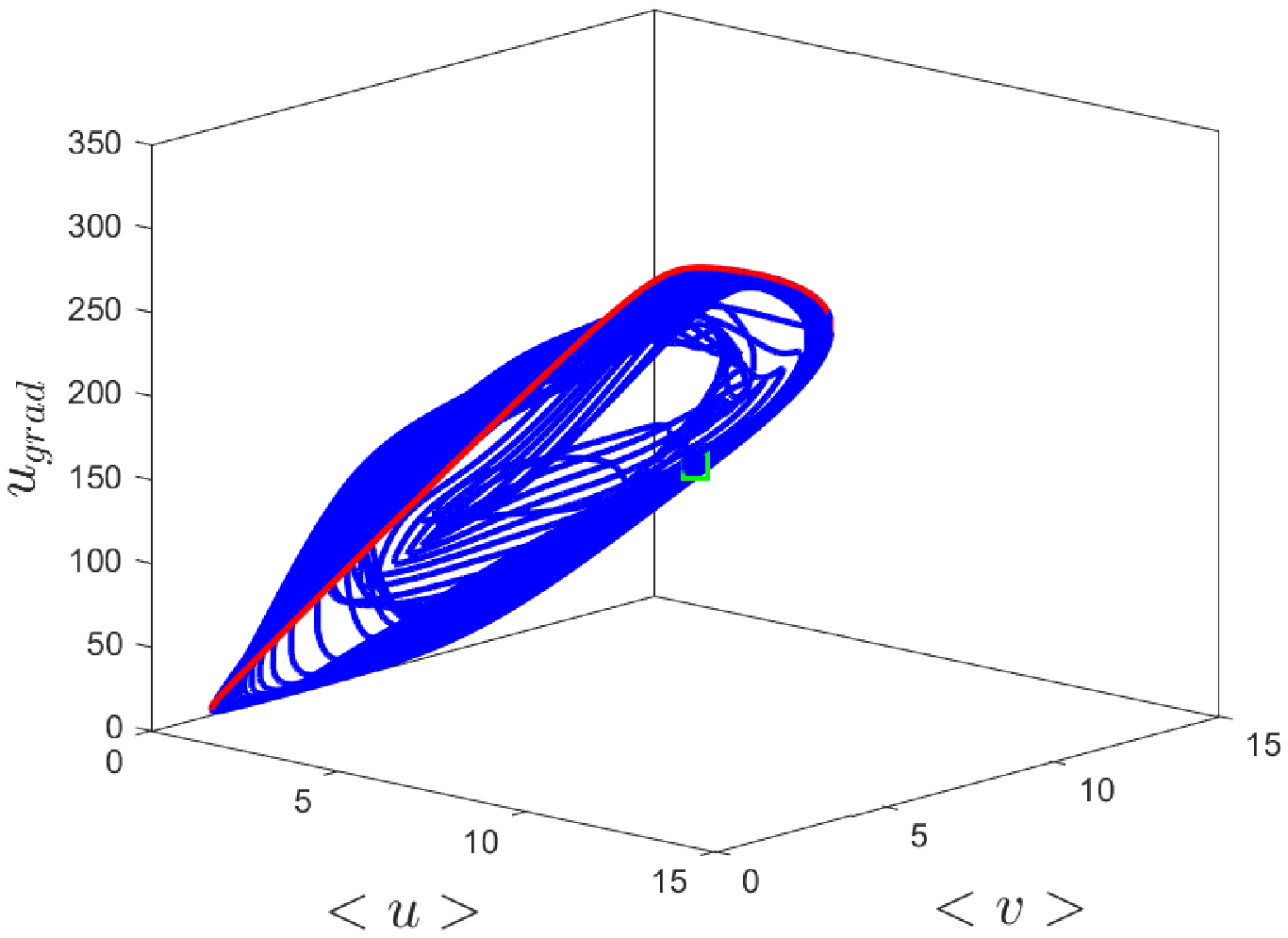}}
		\subfigure[$d=5$]{\includegraphics[scale=0.3]{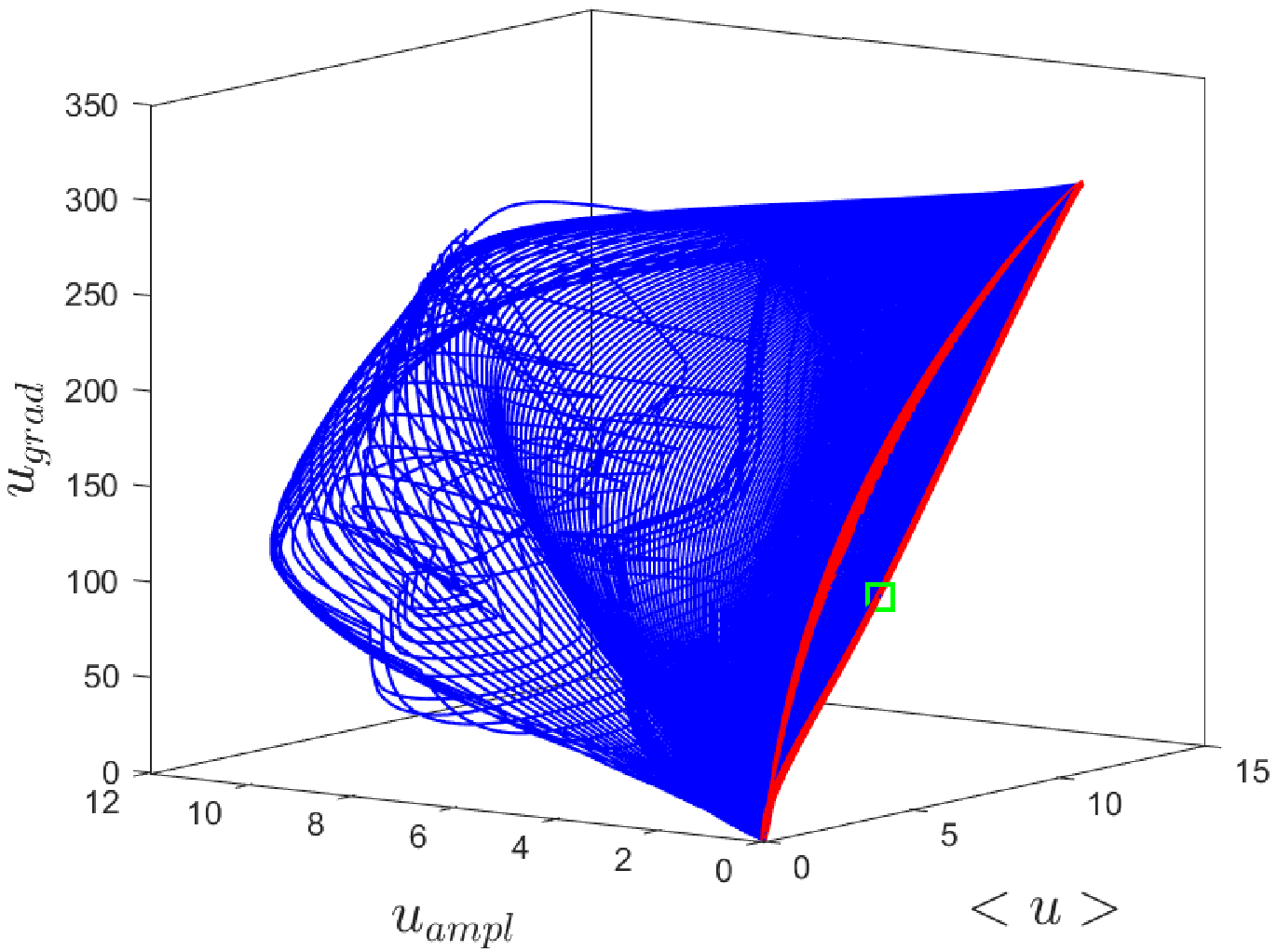}}}
		\mbox{\subfigure[$d=5$]{\includegraphics[scale=0.3]{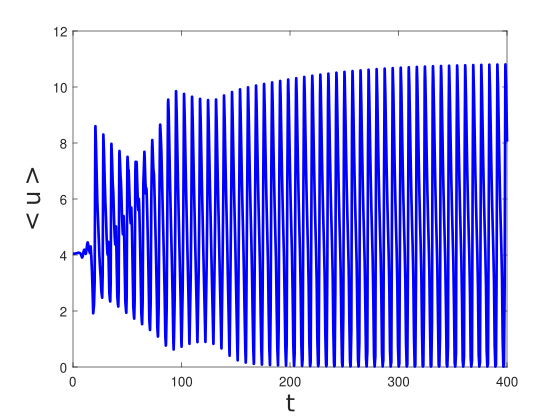}}
		\subfigure[$d=5$]{\includegraphics[scale=0.3]{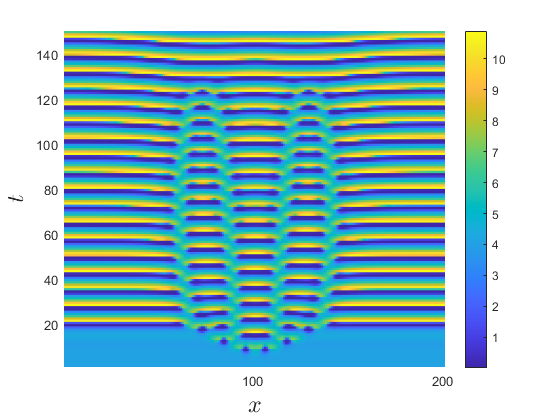}}
	}
	\mbox{\subfigure[$d=25$]{\includegraphics[scale=0.3]{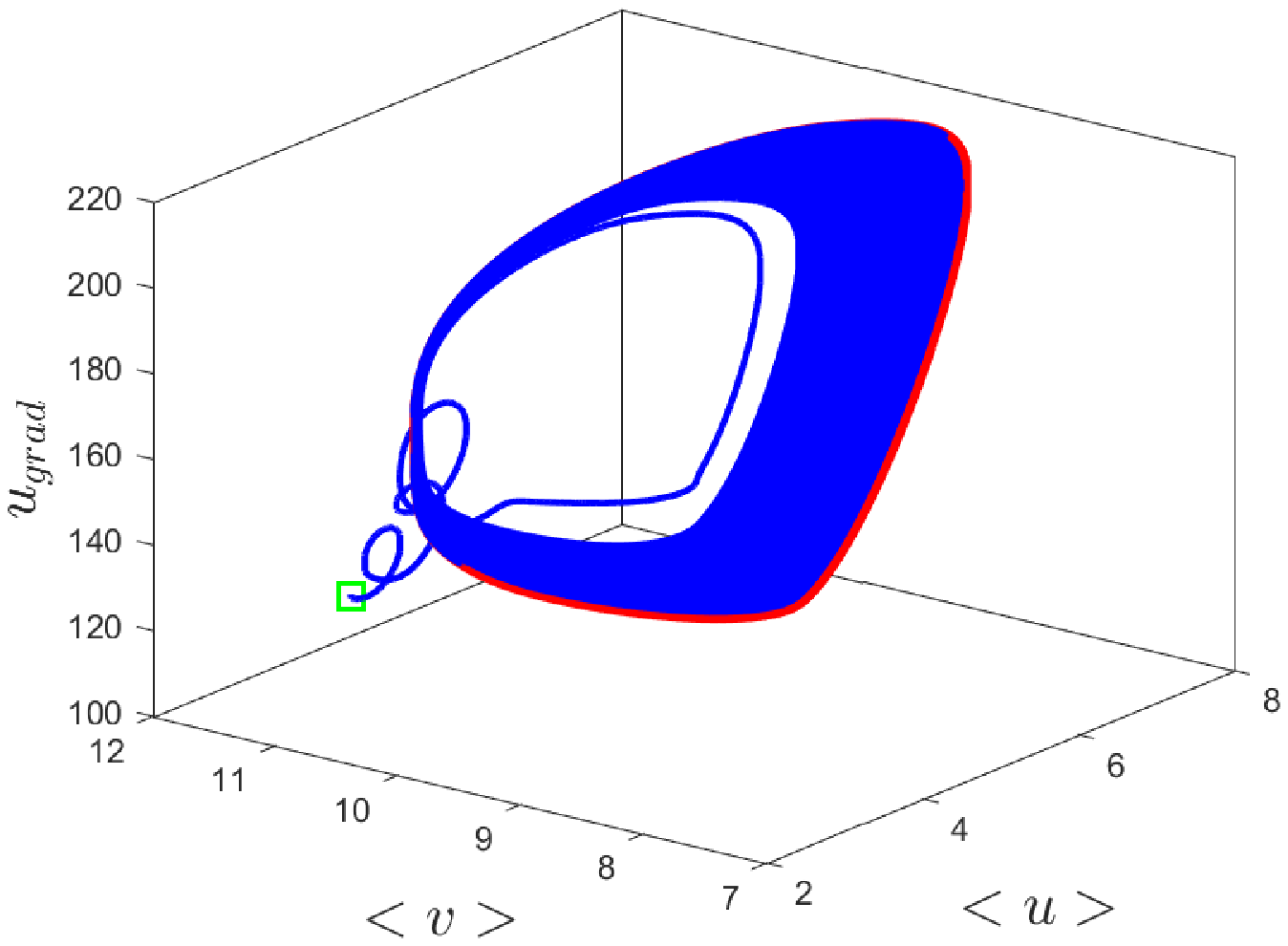}}
		\subfigure[$d=25$]{\includegraphics[scale=0.3]{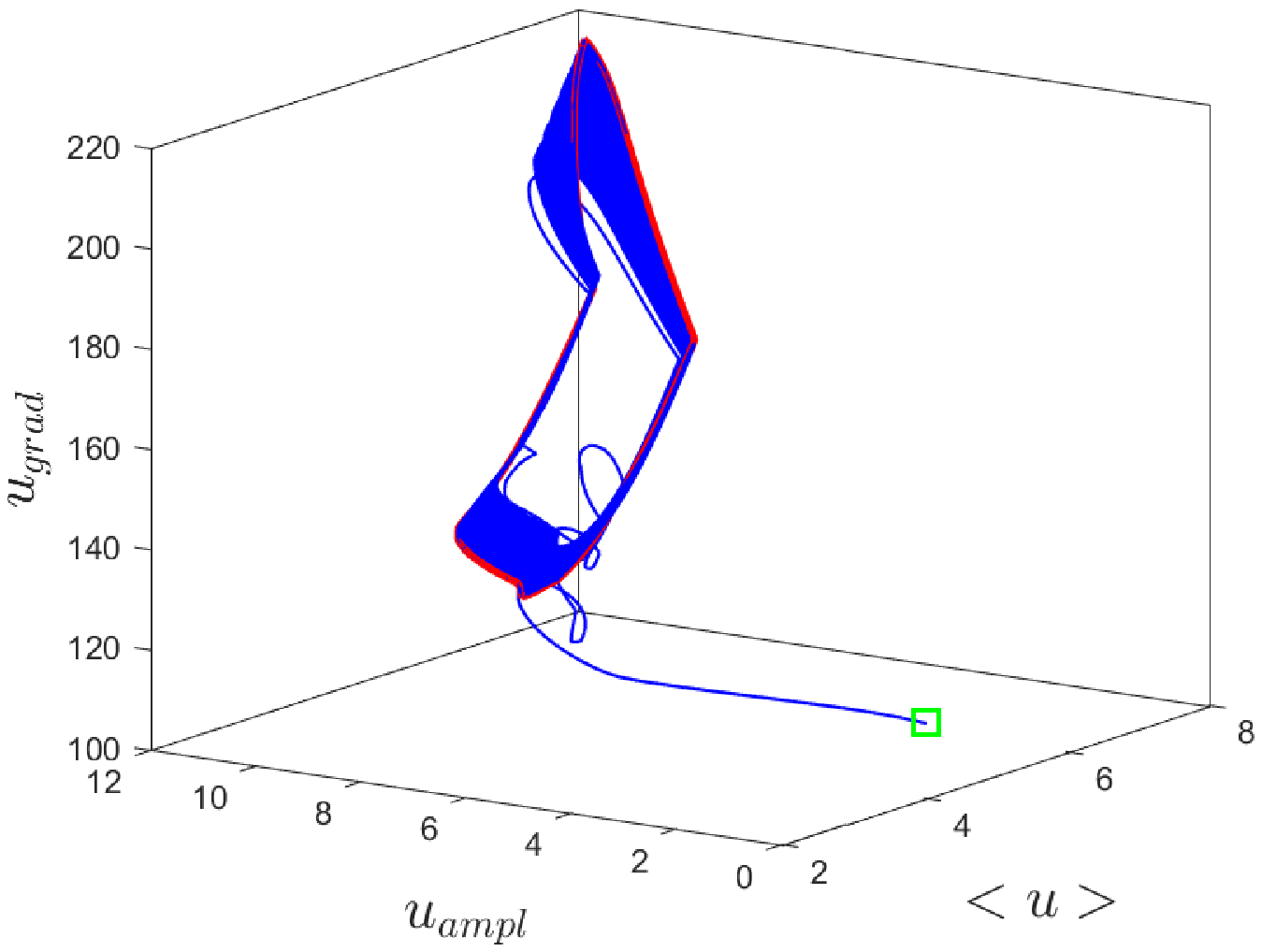}}}
		\mbox{\subfigure[$d=25$]{\includegraphics[scale=0.3]{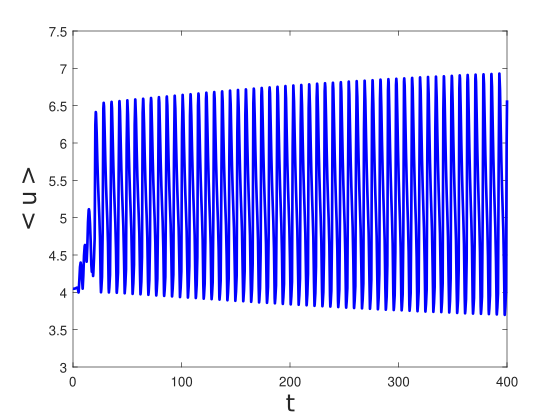}}
		\subfigure[$d=25$]{\includegraphics[scale=0.3]{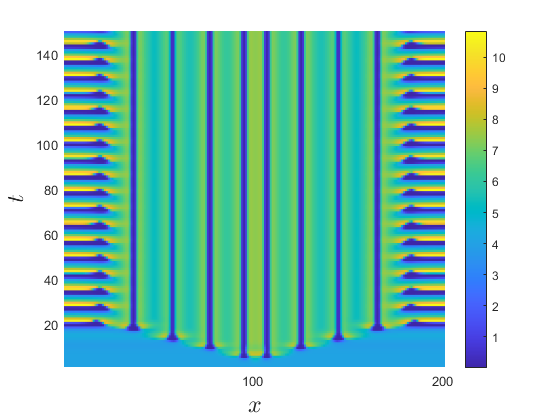}}  
	}
	\caption{Spatio-temporal dynamics of system (\ref{STmodel}) quantified in different ways for $\chi=11.9,\ \delta=0.11,$ and $d=5$ (first and second panels), $d=25$ (third and fourth panels): (a) the plot of $(\langle u\rangle,\langle v\rangle,u_{grad})$, (b) the trajectory of $(\langle u\rangle,u_{ampl},u_{grad}),$ (c) the initial transient dynamics, i.e plot of $(\langle u\rangle,t)$, and (d) the transient patterns. Green and red square mark the starting and ending points respectively. Other parameter values are $\nu=10$, $\alpha=1$, $\beta=2.85$, $\eta=1$, $\varepsilon=1$.}\label{fig:trans_chi11dot9}
\end{figure}

\begin{figure}[ht!]
	\centering
	\mbox{\subfigure[$\varepsilon=0.8$ (upper panel), $\varepsilon=0.6$ (lower panel).]{\includegraphics[scale=0.5]{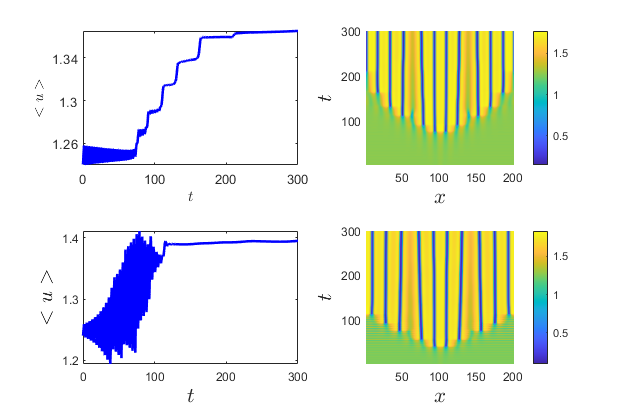}}\hspace{-2em}
	   \subfigure[$\varepsilon=0.8$ (upper panel), $\varepsilon=0.6$ (lower panel).]{\includegraphics[scale=0.5]{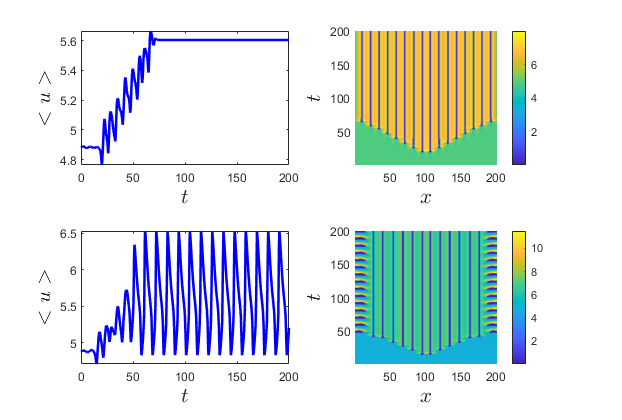}}}
   \mbox{\subfigure[$\varepsilon=0.5$ (upper panel), $\varepsilon=0.1$ (lower panel).]{\includegraphics[scale=0.5]{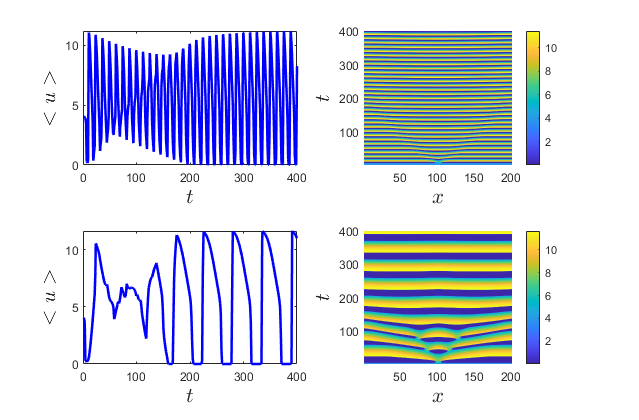}}
}\caption{Comparative study of transients for varying values of $\varepsilon.$ First column of each sub-figure shows the time series plot of spatial average prey density and the second column shows the corresponding spatial distribution of prey density for: (a) $\chi=3.8,\,\, \delta=0.11,\ d=100$, (b) $\chi=12.25,\ \delta=0.11,\ d=10,$ (c) $\chi=11.9,\ \delta=0.11,\ d=1.8$. Other parameter values are fixed as given in text.}\label{Fig:transient_eps1}

\end{figure}

\section{Conclusion}

The classical Rosenzweig-MacArthur (RM) model show two types of coexistence scenarios, a steady state attractor and oscillatory coexistence. The RM model with inclusion of intraspecific competition among the predators is known as Bazykin model \cite{Bazykin}. Introduction of intraspecific competition induces complex dynamical feature exhibited by the system which ranges from bi-stability to extinction through BT-bifurcation \cite{Huang21}. It significantly alters the basin of attraction of coexistence steady state and depending on the initial population densities the stable coexistence state or the oscillatory coexistence is observed. Coexistence of two stable attractors are marked by the unstable limit cycle which is generated through a global bifurcation namely saddle-node bifurcation of limit cycles. Unlike the RM model with slow-fast timescale, the onset of oscillatory dynamics for the Bazykin model depends explicitly on $\varepsilon.$ This dependence is shown on the $\delta-\chi$ parametric domain for different value of $\varepsilon$ in Fig. \ref{fig:bifurcation_schematic}(a). One of the objectives of this work is to reveal the critical relation between the key model parameters, namely, dimensionless carrying capacity of prey population ($\chi$), predator death rate due to intraspecific competition ($\delta$) and the timescale parameter ($\varepsilon$) responsible for a change in the system dynamics.

Most of the analysis already done in this direction are based on the problem which possesses/admits a single stable limit cycle. Hence, the onset of large amplitude oscillation due to canard explosion and existence of relaxation oscillation have received considerable attention over last few years. However, the effect of the slow-fast timescale for the systems which exhibit bi-stability remains poorly explored. Here we study the deformation of stable and unstable limit cycles due to the change in the magnitude of the slow-fast timescale parameter $\varepsilon.$ We have shown analytically the condition for the existence of canard cycles and relaxation oscillation in the considered model. This is validated with the help of numerical examples. The drastic change in the size of limit cycles through canard cycles occur in a very narrow parametric domain which is shown in a schematic bifurcation diagram (cf. \ref{fig:bifurcation_schematic}(b)). Keeping the parameters fixed and choosing $\chi$ to be the bifurcation parameter, we observe that the system exhibits two Hopf bifurcation where one is super critical and the other one is subcritical. The coexistence steady state loses its stability through supercritical Hopf and stable limit cycle appears. In an extremely small neighbourhood of the supercritical Hopf threshold, with the variation of $\varepsilon$ the small Hopf bifurcating limit cycle changes to relaxation oscillation via stable canard cycle with and without head. The coexistence steady state gains stability through subcritical Hopf bifurcation and an unstable limit cycle is formed. The unstable limit cycle changes to unstable canard cycle without head which is surrounded by stable canard cycle with head. Further varying $\varepsilon,$ the stable canard cycle with head changes to stable relaxation oscillation. Thus the slow-fast Bazykin model exhibits two different kinds of canard explosion.  The first kind of canard explosion occurs via a sequence of stable canard cycles due to supercritical Hopf. Whereas the second kind occurs in the parametric range of subcritical Hopf bifurcation and saddle-node bifurcation of limit cycles. Though the size and shape of the limit cycles gets deformed with varying $\varepsilon,$ but the stability of the limit cycles remain unaltered and were computed numerically by considering a slow-divergence integral along the critical manifold.

A spatially explicit system inherits the main properties of the corresponding non-spatial system but can also exhibit additional dynamical behaviors that substantially increase the overall system's dynamical complexity. In particular, the spatial Bazykin's system is capable to produce self-organized spatial and spatio-temporal patterns. We recall here that the classical Rozensweig-McArthur system cannot produce Turing patterns (although it can produce spatio-temporal chaos due to a different mechanism \cite{Medvinsky,Petrovskii03}); it is the intraspecific predator competition (accounted for by the quadratic mortality term) that makes it possible.

In the non-spatial Bazykin's model, the Hopf bifurcation threshold is a function of $\varepsilon$. In the corresponding spatio-temporal model, along with the Hopf, the Turing threshold also depends on $\varepsilon$ and thus the co-dimension 2 Turing-Hopf bifurcation point also shifts with the variation of $\varepsilon$. The size of the Turing domain for fixed range of parameter value shrinks in size (cf. Fig.~\ref{fig:Turing_curve_eps}) and the stationary Turing solution loses its stability forming homogeneous in space and oscillatory in time solution. Whenever the temporal model exhibits bi-stability, the existence of the stationary Turing solution depends on the choice of the initial conditions. Since the spatio-temporal model is infinite dimensional, therefore it is impossible to determine the basin of attraction of the stable homogeneous steady states, stationary Turing patterns and oscillatory in time solutions. Therefore we take the help of extensive numerical simulations to study the stationary as well as dynamic solutions. For the parameter values close to the temporal (Hopf) and spatial (Turing) instability, there is an interference of both the instabilities. This results in long transients before the solution settle down to any self-organizing pattern. The numerical detection of Turing pattern thus becomes even more challenging for parameter values pretty close to the Turing-Hopf threshold. The system may also settle down to stationary pattern after long transient, if the parameter values are chosen from Turing-Hopf domain and close to Turing-Hopf threshold. 

While the large-time (asymptotical) system properties are shaped by the Hopf and Turing bifurcations, the final pattern (or, more generally, final dynamical regime) does not appear until after the initial transients die out. Remarkably, the duration of the initial transients can be very long, in fact infinitely long when $d$ approaches its bifurcation value (cf.~Fig.~\ref{fig:transient_consolidated}). It is this property of the initial transients to become, under certain conditions, very long that makes them particularly relevant to the real-world ecological dynamics \cite{Hastings04,Hastings18}. In this paper, we have shown that the interaction between the predator intraspecific competition and the existence of multiple timescales produces a broad variety of long transients that can last for dozens or even hundreds of generations before the asymptotical pattern takes over; see Figs.~\ref{fig:sptemp_chi3dot8}--\ref{Fig:transient_eps1}.


\section*{Acknowledgements}

The work has been supported by the RUDN University Strategic Academic Leadership Program.




\begin{thebibliography}{99}


\bibitem{Schoener} Schoener, TW. (1983). Field Experiments on Interspecific Competition. {\it The American Naturalist} {\bf 122}(2), 240--285.

\bibitem{Bourlot} Bourlot, V.L., Thomas, T., David, C. (2014). Interference versus Exploitative Competition in the Regulation of Size-Structured Populations. {\it The American Naturalist} {\bf 184}(5), 609-–623. 

\bibitem{Klomp64} Klomp, H. (1964). Intraspecific competition and the regulation of insect numbers. {\it Annual Review of Entomology} {\bf 9}, 17–-40.

\bibitem{Bazykin} Bazykin, A.D. (1998). Nonlinear Dynamics of Interacting Populations, {\it World Scientific}, Singapore.

\bibitem{EPL07} Mcgehee, E.A., Schutt, N., Vasquez, D.A., Peacock-Lopez, E. (2008). Bifurcations, and temporal and spatial patterns of a modified Lotka-Volterra Model. {\it International Journal of Bifurcation and Chaos} {\bf 18}, 2223--2248.

\bibitem{EPL05} Peet, A.B.,  Deutsch, P.A., Peacock-Lopez E. (2005). Complex dynamics in a three-level trophic system with intraspecies interaction. {\it Journal of Theoretical Biology} {\bf 232}, 491–-503.

\bibitem{Muratori89} Muratori, S., Rinaldi, S. (1989). Remarks on competitive coexistence. {\it SIAM Journal on Applied Mathematics} {\bf 49}(5), 1462–-1472.

\bibitem{Fenichel} Fenichel, N. (1979). Geometric singular perturbation theory for ordinary differential equations. { \it Journal of Differential Equations} {\bf 31}, 53–-98.

\bibitem{Wang19} Wang, C., Zhang, X. (2019). Canards, heteroclinic and homoclinic
orbits for a slow-fast predator-prey model of generalized Holling type
IIIs. {\it Journal of Differential Equations} {\bf 267}, 3397-–3441.

\bibitem{Rinaldi92} Rinaldi, S., Muratori, S. (1992). Slow-fast limit cycles in predator–prey models. {\it Ecological Modelling} {\bf 61}, 287–-308.

\bibitem{Poggiale20} Poggiale, J.C., Aldebert, C., Girardot, B., Kooi, B.W. (2020). Analysis of a predator–prey model with specific timescales: a geometrical approach proving the occurrence of canard solutions. {\it Journal of Mathematical Biology} {\bf 80}, 39–-60.

\bibitem{Muratori91} Muratori, S., Rinaldi, S. (1991). A separation condition for the existence of limit
cycles in slow-fast systems. {\it Applied Mathematical Modelling} {\bf 15}, 312--318.

\bibitem{Krupa01b} Krupa, M., Szmolyan, P. (2001a). Relaxation oscillation and Canard explosion. {\it Journal of Differential Equations} {\bf 174}, 312–-368.

\bibitem{Cantrell03} Cantrell, R.S., Cosner, C. (2003). Spatial Ecology via Reaction-Diffusion Equations. {\it John Wiley \& Sons}, England.

\bibitem{Turing52} Turing, A.M. (1952). The chemical basis of morphogenesis. {\it Philosophical Transaction of Royal Society B} {\bf 237}, 37–-72.

\bibitem{Avila17} Avila-Vales, E., Garcia-Almeida, G., Rivero-Esquivel, E. (2017). Bifurcation and spatiotemporal patterns in a Bazykin predator-prey model with self and cross-diffusion and Beddington-DeAngelis response.{\it Discrete and Continuous Dynamical Systems Series B} {\bf 22}(3),  717--740.

\bibitem{Avitabile17} Avitabile, A., Desroches, M., Knobloch, E. (2017). Spatiotemporal canards in neural field equations. {\it Physical Review E} {\bf 95}, 042205.

\bibitem{Hastings04} Hastings, A. (2004). Transients: the key to long-term ecological understanding? {\it TRENDS in Ecology and Evolution} {\bf 19}(1), 39--45.

\bibitem{Hastings18} Hastings, A., Abbott, K.C., Cuddington, K., Francis, T., Gellner, G., Lai Y-C., Morozov, A., Petrovskii, S.V., Scranton, K., Zeeman, M.L. (2018). Transient phenomena in ecology, {\it Science}, {\bf 361} 6406.

\bibitem{Hastings94} Hastings, A., Higgins, K. (1994). Persistence of transients in spatially
structured ecological models. {\it Science} {\bf 263}, 1133-–1136.

\bibitem{Chowdhury21} Chowdhury, P.R., Petrovskii, S., Banerjee, M. Oscillations and pattern formation in a slow–fast prey–predator system, {\it Bulletin of Mathematical Biology}, {\bf 83}, 83:110, 2021

\bibitem{Maesschalck15} Maesschalck, P. De. (2015). Planar Canards with Transcritical Intersections. {\it Acta Applicandae Mathematicae} {\bf 137}, 159–-184. 

\bibitem{Kuehn15} Kuehn, C. (2015). Multiple Time Scale Dynamics. {\it Springer}, New York.

\bibitem{Pao} Pao, C.V. (2012). Nonlinear parabolic and elliptic equations. {\it Springer}, New York.

\bibitem{Camara11} Camara, C.I. (2011). Waves analysis and spatiotemporal pattern formation of an
ecosystem model. {\it Nonlinear Analysis: Real World Applications} {\bf 12} 2511–-2528, 2011.

\bibitem{Mimura80} Mimura, M., Kawasaki, K. (1980). Spatial Segregation in Competitive Interactlon-Diffusion Equations. {\it Journal of Mathematical Biology} {\bf 9}, 49--64.

\bibitem{Murray} Murray, J.D. (2001). Mathematical biology II: spatial models and biomedical applications. {\it Springer}, New York.

\bibitem{Manna} Manna, K., Volpert, V., Banerjee, M. (2021). Pattern Formation in a Three-Species Cyclic Competition Model. {\it Bulletin of Mathematical Biology} {\bf 83}(5), 1--35. 

\bibitem{Mukherjee} Mukherjee, N., Ghorai, S., Banerjee, M. (2019). Cross-diffusion induced Turing and non-Turing patterns in Rosenzweig-MacArthur model. {\it Letters in Biomathematics} {\bf 6}(1), 1--22.

\bibitem{Huang21} Lu, M., Huang, J. (2021). Global analysis in Bazykin’s model with Holling II functional response and predator competition. {\it Journal of Differential Equations} {\bf 280}, 99-–138.

\bibitem{Medvinsky} Medvinsky, A.B., Petrovskii, S.V., Tikhonova, I.A., Malchow, H., Li, B.-L. (2002). Spatiotemporal complexity of plankton and fish dynamics. {\it SIAM Review} {\bf 44}, 311--370.

\bibitem{Petrovskii03} Petrovskii, S.V., Li, B.L., Malchow, H. (2003). Quantification of the spatial aspect of chaotic dynamics in biological and chemical systems. {\it Bulletin of Mathematical Biology} {\bf 65}, 425--446.










\end{thebibliography}
\end{document}